\documentclass[12pt,a4paper]{amsart}

\usepackage[american]{babel}
\usepackage{amsmath}
\usepackage{amssymb}
\usepackage{amsbsy}
\usepackage{amsfonts}
\usepackage{graphicx}
\usepackage{subfigure}
\usepackage[normalem]{ulem}
\usepackage{siunitx}
\usepackage{chemfig}
\usepackage{mhchem}
\usepackage{import}
\usepackage{bbm}

\usepackage{tikz}
\usetikzlibrary{shapes,arrows}
\usepackage{stmaryrd}

\DeclareGraphicsExtensions{.pdf,.jpg,.png,.eps}

\usepackage{url}

\usepackage[format=plain, labelsep=endash, justification=justified, singlelinecheck=true, font={footnotesize,up}]{caption}

\newcommand{\Frac}{\displaystyle\frac}
\newcommand{\jmp}[1]{\,[\![#1]\!]}

\newtheorem{remark}{Remark}[section]
\newtheorem{example}{Example}[section]

\newtheorem{theorem}{Theorem}[section]
\newtheorem{definition}{Definition}[section]
\newtheorem{lemma}{Lemma}[section]

\usepackage{color} 
\definecolor{MidnightBlue}{rgb}	{0,0.1,.5}
\definecolor{darkgreen}{rgb}	{0,0.6,0}

\title[A Stabilized Dual Mixed Hybrid Finite Element Method]
{A Stabilized Dual Mixed Hybrid Finite Element Method with 
Lagrange multipliers for Three-Dimensional 
Problems with Internal Interfaces}

\address{$^{1}$ Dipartimento di Matematica, Politecnico di Milano,  
				 Piazza L. da Vinci 32, 20133 Milano, Italy}
\address{$^{(2)}$ Department of Electrical Engineering and Computer Science, 
College of Engineering, University of Missouri, 201 Naka Hall, Columbia, MO 65211}

\author{Riccardo Sacco$^{1}$ \and Aurelio Giancarlo Mauri$^{1}$ \and
Giovanna Guidoboni$^{2}$}
\email{riccardo.sacco@polimi.it}
\email{aureliogiancarlo.mauri@polimi.it}
\email{guidobonig@missouri.edu}

\begin{document}

\date{\today}

\begin{abstract}
This work focuses on a class of elliptic boundary value problems 
with diffusive, advective and reactive terms, motivated by the study of three-dimensional heterogeneous 
physical systems composed of two or more media separated by a selective interface.
We propose a novel approach for the numerical approximation of such heterogeneous systems 
combining, for the first time: (1) a dual mixed hybrid (DMH) finite element method (FEM)
based on the lowest order Raviart-Thomas space (RT0); (2) a Three-Field (3F) formulation;
and (3) a Streamline Upwind/Petrov-Galerkin (SUPG) stabilization method.
Using the abstract theory for generalized saddle-point problems and their approximation,
we show that the weak formulation of the proposed method
and its numerical counterpart are both uniquely solvable and that the resulting 
finite element scheme enjoys optimal 
convergence properties with respect to the discretization parameter. 
In addition, an efficient implementation of the proposed formulation is presented. 
The implementation is based on a systematic use
of static condensation which reduces the method to a nonconforming finite element approach on a 
grid made by three-dimensional simplices.
Extensive computational tests demonstrate the theoretical conclusions and indicate 
that the proposed DMH-RT0 FEM scheme is accurate and stable even in the presence of marked 
interface jump discontinuities in the solution and its associated normal flux. 
Results also show that in the case of strongly dominating advective terms, 
the proposed method with the SUPG stabilization is capable of resolving 
accurately steep boundary and/or interior layers
without introducing spurious unphysical oscillations or excessive smearing of the solution front.
\end{abstract}

\maketitle


\section{Introduction and motivation}
\label{sec:intro}

The study of heterogeneous physical systems composed of two or more media separated by
selective interfaces is a topic of utmost relevance in applied sciences.
Indeed, many applications in biology~\cite{KEDEM1958,membraneselectivity1961,rutten2002,CANGIANI2010}, 
materials science~\cite{RaoHughesI,RaoHughesII,Lee2011}, 
nanoelectronics~\cite{IEEEchargetrapmodels} and
geophysics~\cite{Martin2005,2016APS}, to name a few, are characterized by interface phenomena that play a crucial role in determining the transmission of physical quantities between different media and/or between different regions within the same medium.

The present work focuses on a class of mathematical problems directly motivated by the aforementioned applications. Specifically, we consider a stationary advection-diffusion-reaction problem in a three-dimensional  volume, denoted by $\Omega\subset\mathbb R^3$, whose physical properties may vary in space, thereby leading to an elliptic second-order partial differential equation with variable coefficients. In addition, we account for the presence of a selective internal interface, denote by $\Gamma$, which is geometrically represented by a two-dimensional manifold in $\Omega$ and on which we impose suitable transmission conditions  to ensure the balance of flux density across the interface and to model segregation phenomena that may occur within the interface itself. 
For example, the mathematical setting considered in this article may be used to describe
superficial chemical processes involved in semiconductor crystal growth~\cite{Causin2004} 
or mass transport and reaction mechanisms occurring at the cellular scale across 
the membrane lipid bilayer~\cite{Wood2002}.

The fact that many driving processes actually occur at internal interfaces poses serious challenges for the numerical solution of the class of problems described above. 
In particular, in order to obtain physically-relevant solutions it is crucial to maintain the main physical features associated with interfacial phenomena from the continuous to the discrete level, including the continuity of flux density at the interface. Many numerical approaches have been proposed for the solution of elliptic problems in spatially heterogeneous domains. In particular, domain decomposition methods have been proven to be very effective in dealing with partitions in the volume, which may result from physical heterogeneities in the medium and/or from artificial partitioning aimed at reducing the computational costs of large-scale problems. Many different discretization techniques have been utilized within the context of domain decomposition methods, including finite elements, spectral elements and finite volumes. 
We refer to~\cite{quarteroni1999domain} for a complete overview of theoretical and computational properties of the domain decomposition approach. 

Motivated by the need of accurately capturing interface phenomena, in this work we propose a novel numerical approach that combines, for the first time: 
\begin{enumerate}
\item a \textit{Dual Mixed Hybrid (DMH) finite element method (FEM)} in order to ensure that: 
\textit{(i)} the solution (or primal variable) verifies the given partial differential equation
 within each element (see~\cite{RaviartThomas1977});
 \textit{(ii)} the flux (or dual variable) associated with the solution is continuous across elements (see~\cite{deVeubeke1965,RaviartThomas1979});
and 
\textit{(iii)} both primal and dual variables satisfy optimal error estimates (see~\cite{brezzifortin1991,RobertsThomas1991});
\item  a \textit{Three-Field formulation (3F)}, typical of domain decomposition approaches, in order to account for interfacial discontinuities within the weak formulation of the problem (see~\cite{brezzi3F1994,quarteroni1999domain,BrezziMarini3F_MathComp_2000}); 
\item a \textit{Streamline Upwind/Petrov-Galerkin (SUPG) stabilization method} in order to gain the required amount of numerical stability without 
significantly spoiling the accuracy of the computed solution due to excessive 
crosswind smearing (see~\cite{BROOKS1982,HUGHES1987}). 
\end{enumerate}
We remark that the 
pair of Lagrange multipliers introduced within the 3F formulation is a natural fit for the DMH FEM functional framework (see~\cite{brezzi3F1994,quarteroni1999domain}). 
In addition, the use of static condensation allows us to 
eliminate variables defined in the interior of each element in favor of the sole hybrid variable, thereby
obtaining a final algebraic system structurally analogous to that of a standard primal-based
finite element approach (see~\cite{ArnoldBrezzi1985} and~\cite[Chapter~5]{brezzifortin1991}).

The proposed stabilized DMH-RT0 FEM scheme is analyzed at both the infinite and finite dimensional levels 
using the abstract theory of saddle-point problems; its well-posedness and optimal error estimates are proved under suitable assumptions on the data.
A series of simulations is performed to validate the accuracy and robustness of the novel method via comparison between numerical and analytical solutions in three-dimensional test cases.
Results show that the proposed stabilized DMH-RT0 FEM scheme ({\em i}) satisfies the theoretical findings 
even in the presence of marked interface jump discontinuities in the solution and its associated flux;
and ({\em ii}) is capable of accurately resolving steep boundary and/or interior layers
without introducing spurious unphysical oscillations or excessive smearing of the solution front.

An overview of the article is as follows.
Section~\ref{sec:model} introduces the mathematical model
and 
the physical meaning of interface and
boundary conditions.
Section~\ref{sec:dmh_method} presents the weak formulation of the problem through
the novel DMH method proposed in the article and the analysis
of its well-posedness through the general theory reported in~\ref{sec:saddle_point_continuous}.
Section~\ref{sec:dmh_Galerkin_FE_approximation} presents the Galerkin approximation of
the DMH weak problem studied in Section~\ref{sec:dmh_method} and  
the analysis of its well-posedness through the general theory 
reported in~\ref{sec:saddle_point_approx}.
Section~\ref{sec:static_condensation} addresses the issue of how to efficiently implement 
the proposed DMH-RT0 FEM scheme via static condensation whereas Section~\ref{sec:stabilization}
describes how to introduce a mechanism of stabilization into the DMH-RT0 FEM scheme to prevent 
the onset of spurious unphysical oscillations when the problem becomes advection-dominated.
Section~\ref{sec:spectral_analysis_stab} is devoted to the spectral analysis of the stabilized diffusion tensor. Section~\ref{sec:numerical_results} provides a thorough discussion of the numerical simulations
conducted to validate the accuracy and stability of the novel DMH-RT0 FEM scheme.
Section~\ref{sec:conclusions} gives a summary of the content of the work and an overview of future
investigations. 

\section{Mathematical model}
\label{sec:model}

Let $\Omega$ be an open polyhedral subset of $\mathbb{R}^3$ and let $\partial \Omega \equiv \Sigma$
denote the boundary of $\Omega$ on which an outward unit normal vector $\mathbf{n}$ is defined
(see Figure~\ref{fig:domain}).
\begin{figure}[h!]
\centering
\includegraphics[width=0.45\textwidth]{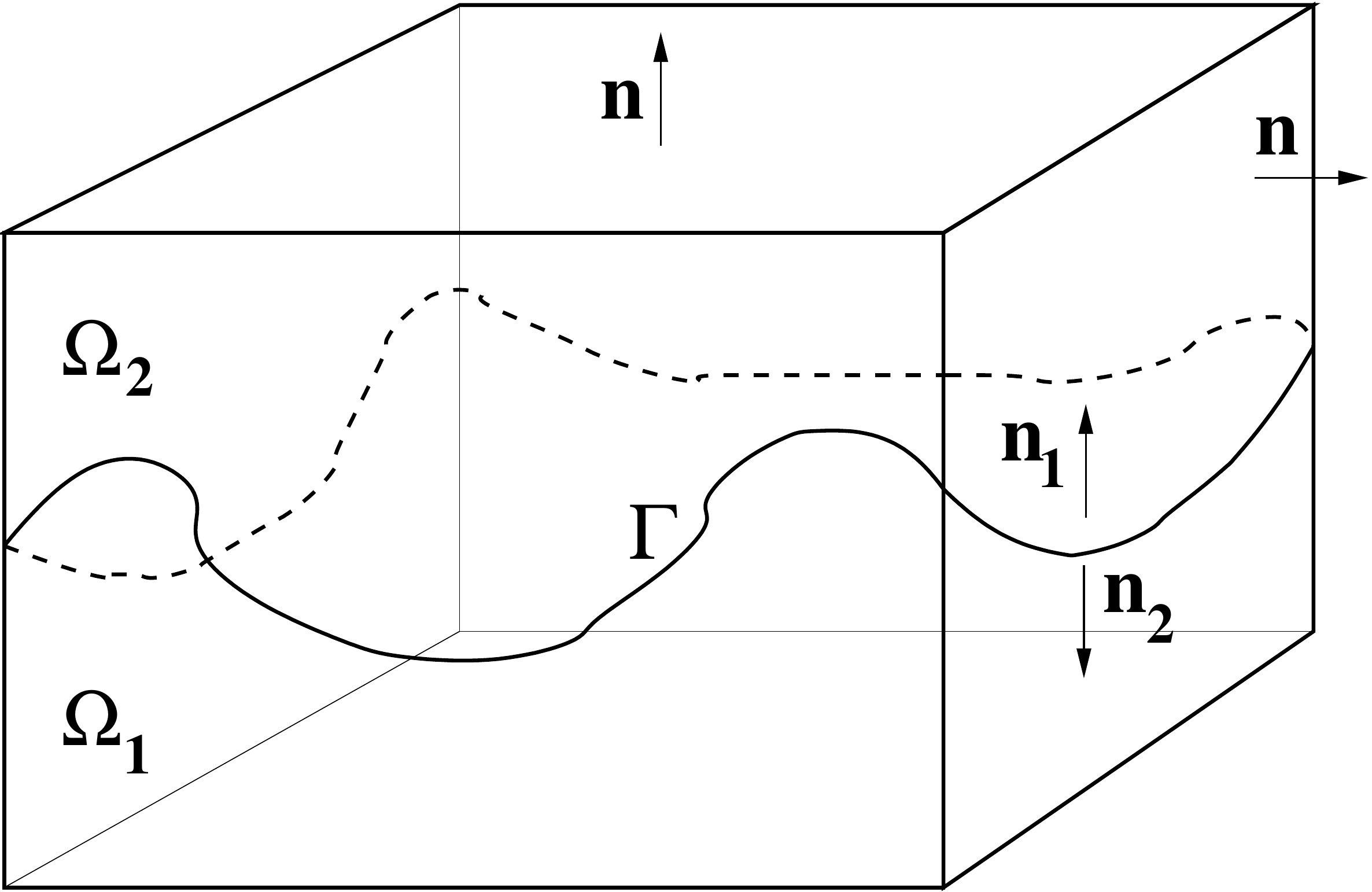}
\caption{The domain $\Omega$, its partition into subregions $\Omega_1$, $\Omega_2$, the internal interface 
$\Gamma$ and the geometrical notation.}
\label{fig:domain}
\end{figure}
The domain $\Omega$ is the union of two subregions $\Omega_1$ and $\Omega_2$, whose boundaries 
are denoted by $\partial \Omega_1$ and $\partial \Omega_2$, respectively. The two subregions are separated by the interface $\Gamma = \partial \Omega_1 \cap \partial \Omega_2$. 
For any function $w:\Omega \rightarrow \mathbb{R}$, we denote by $w_1$ and $w_2$ the restrictions of $w$
to $\Omega_1$ and $\Omega_2$, respectively. We also denote by $w_1|_\Gamma$ and $w_2|_\Gamma$
the traces on $\Gamma$ of $w_1$ and $w_2$, respectively.
For each point $\mathbf{y} \in \Gamma$, we define two
unit normal vectors $\mathbf{n}_1(\mathbf{y})$ and $\mathbf{n}_2(\mathbf{y})$ outwardly directed with respect to $\Omega_1$ and $\Omega_2$, respectively,
for which it holds $\mathbf{n}_1(\mathbf{y}) + \mathbf{n}_2(\mathbf{y})=\mathbf{0}$.
Thus, the three-dimensional problem considered in this article reads:
\begin{subequations}\label{eq:tx_model}
\begin{align}
& {\rm div} \mathbf{J} + r u = g         & \textrm{in} \, \, (\Omega \setminus \Gamma) \label{eq:continuity} \\
& \mathbf{J} = \mathbf{v} u - \boldsymbol{\mu} \nabla u & \textrm{in} \, \, (\Omega \setminus \Gamma) \label{eq:J} \\
& \mathbf{J}_1|_\Gamma \cdot \mathbf{n}_1 + \mathbf{J}_2|_\Gamma \cdot \mathbf{n}_2
= -\sigma & \textrm{on} \, \, \Gamma \label{eq:tx_flux} \\
& u_2|_{\Gamma} = \kappa u_1|_{\Gamma} &  \textrm{on} \, \, \Gamma \label{eq:tx_segr} \\
& \gamma \mathbf{J} \cdot \mathbf{n} = \alpha u - \beta & \textrm{on} \, \, \Sigma. &
\label{eq:bcs}
\end{align}

The dependent variables of the problem are $u$ and $\mathbf{J}$. 
In the remainder of the article,
we shall refer to $u$ as the \emph{primal} variable and to $\mathbf{J}$ as
the \emph{dual} variable. The meaning of this terminology is related to the variational principles 
associated with the solution of system~\eqref{eq:tx_model} (see~\cite{brezzifortin1991}). 
Equation~\eqref{eq:continuity} is a stationary conservation law in which the quantity 
$g - ru$  represents a net production rate of the physical quantity modeled by the function $u$,
with $r=r(\mathbf x)$ and $g=g(\mathbf x)$ denoting nonnegative and bounded given functions of space.
The given advection field  $\mathbf{v}=\mathbf v(\mathbf x)$ is assumed to be piecewise smooth over $\Omega$, 
whereas the diffusivity 
tensor $\boldsymbol{\mu}$ is assumed to be a multiple of the identity, namely $\boldsymbol{\mu} (\mathbf x)= \mu (\mathbf x) \mathbf{I}$, where
$\mathbf{I}$ is the identity tensor in $\mathbb{R}^3$ and the function $\mu$ satisfies the
following bound
\begin{align}
& 0 < \mu_{min} \leq \mu(\mathbf{x}) \leq \mu_{max} < +\infty
& \textrm{for a.e.} \,
\mathbf{x} \in \Omega. & \label{eq:mu_bound}
\end{align}

Equations~\eqref{eq:tx_flux} and~\eqref{eq:tx_segr} are the transmission conditions enforced on
the interface $\Gamma$. Equation~\eqref{eq:tx_flux} expresses the balance of flux density across the interface separating the two subdomains, where
the given function $\sigma=\sigma(\mathbf x)$ represents
a superficial source or sink over the interface. 
Equation~\eqref{eq:tx_segr} expresses the mechanism of segregation occurring within the interface, where the nonnegative function $\kappa=\kappa(\mathbf x)$ represents a local equilibrium constant~\cite{Wood2002}. In particular, if $\kappa=1$ and $\sigma=0$, problem~\eqref{eq:tx_model} corresponds to a
multidomain formulation of the advection-diffusion-reaction equation~\eqref{eq:continuity}-~\eqref{eq:J}
over the \emph{whole} domain $\Omega$.
Equation~\eqref{eq:bcs} expresses the boundary condition on the external surface of $\Omega$, where $\alpha=\alpha(\mathbf x)$, $\beta=\beta(\mathbf x)$ and $\gamma=\gamma(\mathbf x)$ are given functions. In particular, we assume that 
\begin{align}
& \alpha (\mathbf x) >0
\quad \mbox{and} \quad
0 \leq \gamma(\mathbf{x}) \leq 1
& \textrm{for a.e.} \, \mathbf{x} \in \Sigma. & \label{eq:gamma_bounds}
\end{align}
We remark that Equation~\eqref{eq:bcs} corresponds to a Robin boundary condition in the case $\gamma >0$ and to a Dirichlet boundary condition in the case $\gamma=0$.
\end{subequations}
For the sake of simplicity, in the remainder of the article 
(with the sole exception of Section~\ref{sec:numerical_results}), 
we assume $\kappa$ to be a positive constant and $\gamma =1$.

\section{Dual mixed hybrid weak formulation}\label{sec:dmh_method}
The weak formulation of problem~\eqref{eq:tx_model} is obtained by extending the DMH method (see~\cite{Thomas1977thesis,RaviartThomas1979,farhloul1993,Farhloul1997,fortin2017}) to include Lagrange multipliers for the interface conditions~\eqref{eq:tx_flux} and~\eqref{eq:tx_segr}, in the spirit of the 3F formulation (see~\cite{brezzi3F1994,quarteroni1999domain,BrezziMarini3F_MathComp_2000}). For the sake of clarity, we begin by describing 
the functional setting in Section~\ref{sec:functional_setting_dmh}, 
followed by 
the geometrical discretization of the domain in Section~\ref{sec:geometrical_discretization},
the derivation of the weak formulation in  Section~\ref{sec:dmh_weak_form} and the study of its well-posedeness
in Section~\ref{sec:well-posedness_dmh}.

\subsection{Functional setting}\label{sec:functional_setting_dmh}

Let us denote by $\mathcal{S}$ an open bounded subset of $\mathbb{R}^3$ 
having a boundary $\partial \mathcal{S}$. Throughout the article, we will utilize the functional spaces 
$L^2(\mathcal S)$, $H^1(\mathcal S)$ and $H({\rm div}\mathcal S)$,
endowed with the usual $L^2-$, $H^1-$ and $H({\rm div})-$ norms denoted by 
$\|\cdot\|_{0,\mathcal S}$, $\|\cdot\|_{1,\mathcal S}$ and $\| \cdot \|_{H({\rm div}; \mathcal{S})}$, 
respectively, with div denoting the divergence operator. 
We will also utilize the trace theorems, which involve 
the functional space $H^{1/2}(\partial \mathcal S)$ and its dual $H^{-1/2}(\partial \mathcal S)$ 
endowed with the norms:
\begin{subequations}\label{eq:functional_spaces}
\begin{align} 
&  \displaystyle \| \eta \|_{1/2, \partial \mathcal S} = 
\inf_{\substack{\phi \in H^1(\mathcal S)\\ \phi|_{\partial \mathcal S}=\eta}} \| \phi \|_{1,\mathcal S}, 
& \label{eq:H1half_norm} \\
& \displaystyle \| \mu\|_{-1/2,\partial \mathcal S} = 
\inf_{\substack{\mathbf{q} \in H({\rm div};\mathcal S)\\ \mathbf{q} \cdot \mathbf{n}|_{\partial\mathcal S}=\mu}}
\| \mathbf{q} \|_{H({\rm div};\mathcal S)}. & \label{eq:H-1half_norm}
\end{align}
\end{subequations}
We refer to~\cite{Thomas1977thesis,RobertsThomas1991,brezzifortin1991} and references cited therein 
for definitions and mathematical properties of the above mentioned functional spaces.
In addition, we will denote by $(\cdot, \cdot)_{\mathcal S}$ the scalar product in $L^2$ over $\mathcal S$ 
and, for simplicity, we will use the shortened notation
$(\cdot, \cdot)_i$ for the scalar product in $L^2$ over $\Omega_i$.

\subsection{Geometrical discretization}\label{sec:geometrical_discretization}
Let $\left\{ \mathcal{T}_h \right\}_{h>0}$ denote a family of regular triangulations 
of the computational domain $\Omega$ made of 
closed tetrahedral elements $K$ (cf. Definition~3.4.1 of~\cite{quarteroni1994numerical}), where
the positive quantity $h$ represents the discretization parameter. 
We assume that each partition of the family satisfies the admissibility 
criteria of~\cite{quarteroni1994numerical}, Section 3.1. We also assume that 
each subdomain $\Omega_i$, $i=1,2$, is exactly covered 
by the elements of $\mathcal{T}_h$ and we denote by $\mathcal{T}_{h,1}$ 
and $\mathcal{T}_{h,2}$ the restrictions of $\mathcal{T}_{h}$ to $\Omega_1$ and $\Omega_2$,
in such a way that $\Omega = \mathcal{T}_{h,1} \cup \mathcal{T}_{h,2}$ and 
$\Gamma = \mathcal{T}_{h,1} \cap \mathcal{T}_{h,2}$. This latter property amounts to assuming that
the two partitions connect in a \emph{conforming} manner at the interface. 
For more general geometrical approaches and related numerical schemes, 
we refer to~\cite{quarteroni1999domain} in the context of domain
decomposition methods, to~\cite{CockburnGopalakrishnanLazarov2009} in the context of 
Hybridizable Discontinuous Galerkin finite elements and to~\cite{xfem_2010} in the context of
Extended finite element methods.

For every $K \in \mathcal{T}_h$, we denote by $h_K$ the diameter of $K$ and we let $h:= \displaystyle 
\max_{K \in \mathcal{T}_h} h_K$. We denote by $\partial K$ the boundary of $K$ and
by $\mathbf{n}_{\partial K}$ the outward unit normal vector on $\partial K$.
For each pair of neighbouring elements $K_1$ and $K_2$ belonging to $\mathcal{T}_h$, we 
define their common face as $F:= \partial K_1 \cap \partial K_2$. Correspondingly, we introduce
the following sets of faces:
\begin{itemize}
\item $\mathcal{F}_h$: the set of faces belonging to $\mathcal{T}_h$; 
\item $\mathcal{F}_{h,int}$: the subset of faces 
belonging to the interior of $\Omega$ but not to $\Gamma$; 
\item $\mathcal{F}_{h,\Gamma}$: the subset of faces 
belonging to $\Gamma$; 
\item $\mathcal{F}_{h,\Sigma}$: the set of faces belonging
to the domain boundary $\Sigma$. 
\end{itemize}
The set $\mathcal{F}_{h,int}$ can be divided into the sum of the
two disjoint sets $\mathcal{F}_{h,int,1}$ (faces in the interior of $\Omega_1$) 
and $\mathcal{F}_{h,int,2}$ (faces in the interior of $\Omega_2$). Analogously, 
$\mathcal{F}_{h,\Sigma}$ can be divided into the sum of the
two disjoint sets $\mathcal{F}_{h,\Sigma_1}$ (faces on $\Sigma_1$) and $\mathcal{F}_{h,\Sigma_2}$
(faces on $\Sigma_2$). According to these decompositions we have:
\begin{subequations}\label{eq:mesh_sets}
\begin{align}
& \mathcal{F}_{h,int} = \mathcal{F}_{h,int,1} \cup \mathcal{F}_{h,int,2}, & \\
& \mathcal{F}_{h,\Sigma} = \mathcal{F}_{h,\Sigma_1} \cup \mathcal{F}_{h,\Sigma_2}, & \\
& \mathcal{F}_h = \mathcal{F}_{h,int} \cup \mathcal{F}_{h,\Gamma} \cup \mathcal{F}_{h,\Sigma}.
\end{align}
\end{subequations}
We also define the sets:
\begin{subequations}\label{eq:double_faces_sets}
\begin{align}
& \mathcal{F}_{h,1} = \mathcal{F}_{h,int,1} \cup \mathcal{F}_{h,\Sigma_1} \cup \mathcal{F}_{h,\Gamma}, & \\
& \mathcal{F}_{h,2} = \mathcal{F}_{h,int,2} \cup \mathcal{F}_{h,\Sigma_2} \cup \mathcal{F}_{h,\Gamma}, & \\
& \mathcal{F}_{h,\Gamma,1} = \mathcal{F}_{h,1} \cap \mathcal{F}_{h,\Gamma}, & \\
& \mathcal{F}_{h,\Gamma,2} = \mathcal{F}_{h,2} \cap \mathcal{F}_{h,\Gamma}. &
\end{align}
\end{subequations}

\subsection{The DMH weak formulation}\label{sec:dmh_weak_form}

For every set $\mathcal{S} \in \mathbb{R}^3$, let us introduce the following subspace of 
$H({\rm div};\mathcal{S})$
\begin{align}
& \mathcal H({\rm div}; \mathcal{S}):= \left\{ \mathbf{q} \in H({\rm div};\mathcal{S}) \, | \, 
\mathbf{q} \cdot \mathbf{n}_{\partial \mathcal{S}} \in L^2(\partial \mathcal{S}) 
\right\} \subset H({\rm div};\mathcal{S}). & \label{eq:Hcaldiv}
\end{align}
Then, we introduce the following spaces on the partitioned triangulation:
\begin{subequations}\label{eq:dmh_spaces}
\begin{align}
& \mathbf{V}_i = \left\{ \mathbf{v} \in (L^2(\Omega_i))^3, \, \mathbf{v}_K \in \mathcal{H}({\rm div}; K) 
\, \forall K \in \mathcal{T}_{h,i} \right\}, & \label{eq:W_q_div_Omega_i_dmh} \\
& V_i = L^2(\Omega_i), & \label{eq:V_space_i_dmh} \\
& M_i = \left\{ \mu \in  L^2(\mathcal{F}_{h,i}), \, \mu_{K_1} = \mu_{K_2} \, \, 
\textrm{on } \, \partial K_1 \cap \partial K_2, \, K_1, K_2 \in \mathcal{T}_{h,i} \right\}, &
\label{eq:M_space_uhat_dmh} \\
& M_{J,i} =  L^2(\mathcal{F}_{h,\Gamma,i}), & \label{eq:M_space_Jcal_dmh} \\
& M_{\lambda} = L^2(\mathcal{F}_{h,\Gamma}). & \label{eq:M_space_lambda_dmh}
\end{align}
\end{subequations}

\begin{remark}\label{rem:double_faces_Gamma}
Functions in $M_1$ and $M_2$ are single-valued on each face belonging to the interior
of $\mathcal{T}_{h,1}$ and $\mathcal{T}_{h,2}$ and on each face belonging to
$\Sigma_1$ and $\Sigma_2$. On the contrary, on each face $F$ belonging to $\mathcal{F}_{h,\Gamma}$ 
we have, in general, $\mu_1|_F \neq \mu_2|_F$, $\mu_1 \in M_1$, $\mu_2 \in M_2$. 
This is the reason why
the faces on $\Gamma$ are attributed
to both sets $\mathcal{F}_{h,1}$ and $\mathcal{F}_{h,2}$ in the definition~\eqref{eq:double_faces_sets}.
The same argument holds for functions belonging to the spaces $M_{J,1}$ and $M_{J,2}$.
\end{remark}

We set $\mathcal{V}:= \mathbf{V}_1 \times V_1 \times \mathbf{V}_2 \times V_2$, 
$\mathcal{Q}:= M_1 \times M_2 \times M_{J,1} \times M_{J,2} \times M_{\lambda}$, and we define
${\tt u}:= (\mathbf{J}_1, u_1, \mathbf{J}_2, u_2) \in \mathcal{V}$, 
${\tt p}:= (\widehat{u}_1, \widehat{u}_2, \mathcal{J}_1, \mathcal{J}_2, \lambda) \in \mathcal{Q}$,
${\tt v}:= (\boldsymbol{\tau}_1, \phi_1, \boldsymbol{\tau}_2, \phi_2) \in \mathcal{V}$ and
${\tt q}:= (\mu_1, \mu_2, \rho_1, \rho_2, \varphi) \in \mathcal{Q}$. 
In the sequel, ${\tt u}$ will represent the vector of the unknowns defined in the interior of each 
mesh element, ${\tt p}$ will represent the vector of the unknowns defined on the faces of the
domain partition whereas ${\tt v}$ and ${\tt q}$ will represent the vectors of the 
test functions belonging to $\mathcal{V}$ and $\mathcal{Q}$, respectively.

Based on definitions~\eqref{eq:dmh_spaces}, 
we endow $\mathcal{V}$ and $\mathcal{Q}$ with the following norms:
\begin{subequations}\label{eq:norm_spaces_dmh}
\begin{align}
& \| {\tt u} \|_{\mathcal{V}} = 
\Big(\sum_{K \in \mathcal{T}_{h,1}} \| \mathbf{J}_1 \|_{H({\rm div};K)}^2 + 
\sum_{K \in \mathcal{T}_{h,2}} \| \mathbf{J}_2 \|_{H({\rm div};K)}^2 + 
\| u_1 \|_{0,\Omega_1}^2 + \| u_2 \|_{0,\Omega_2}^2\Big)^{1/2}, \label{eq:norm_V_dmh} \\
& \| {\tt p} \|_{\mathcal{Q}} = 
\Big(\sum_{F \in \mathcal{F}_{h,1}} \| \widehat{u}_1 \|_{0,F}^2 
+ \sum_{F \in \mathcal{F}_{h,2}} \| \widehat{u}_2 \|_{0,F}^2
+ \sum_{F \in \mathcal{F}_{h,\Gamma,1}} \| \mathcal{J}_1 \|_{0,F}^2 \nonumber \\
& \qquad \qquad + \sum_{F \in \mathcal{F}_{h,\Gamma,2}} \| \mathcal{J}_2 \|_{0,F}^2 
+ \sum_{F \in \mathcal{F}_{h,\Gamma}} \| \lambda \|_{0,F}^2\Big)^{1/2}. \label{eq:norm_Q_dmh}
\end{align}
\end{subequations}
For all ${\tt u} \in \mathcal{V}$, ${\tt v} \in \mathcal{V}$, and for all
${\tt p} \in \mathcal{Q}$, ${\tt q} \in \mathcal{Q}$, we introduce the following bilinear forms:
\begin{subequations}\label{eq:bilinear_forms_dmh}
\begin{align}
 a({\tt u}, {\tt v}) : = & 
(\boldsymbol{\mu}^{-1} \mathbf{J}_1, \boldsymbol{\tau}_1)_1 
- (\boldsymbol{\mu}^{-1} \mathbf{v} u_1, \boldsymbol{\tau}_1)_1 
+ (r u_1,\phi_1)_1 \nonumber \\
& - \sum_{K \in \mathcal{T}_{h,1}} (u_1, {\rm div}\boldsymbol{\tau}_1)_K 
+ \sum_{K \in \mathcal{T}_{h,1}} (\phi_1, {\rm div} \mathbf{J}_1)_K \nonumber \\
& + (\boldsymbol{\mu}^{-1} \mathbf{J}_2, \boldsymbol{\tau}_2)_2 
- (\boldsymbol{\mu}^{-1} \mathbf{v} u_2, \boldsymbol{\tau}_2)_2 
+ (r u_2, \phi_2)_2 \nonumber \\
& - \sum_{K \in \mathcal{T}_{h,2}} (u_2, {\rm div}\boldsymbol{\tau}_2)_K 
+ \sum_{K \in \mathcal{T}_{h,2}} (\phi_2, {\rm div} \mathbf{J}_2)_2, \label{eq:bilinear_A_dmh} \\
 b({\tt u}, {\tt q}): = & 
\sum_{K \in \mathcal{T}_{h,1}} 
(\mathbf{J}_1 \cdot \mathbf{n}_{\partial K}, \mu_1)_{\partial K} + 
\sum_{K \in \mathcal{T}_{h,2}} (\mathbf{J}_2 \cdot \mathbf{n}_{\partial K}, 
\mu_2 )_{\partial K}, \label{eq:bilinear_B_dmh} \\
 c({\tt p}, {\tt q}): = &  
\sum_{F \in \mathcal{F}_{h,\Sigma_1}} (\alpha \widehat{u}_1, \mu_1)_F
+ \sum_{F \in \mathcal{F}_{h,\Sigma_2}} (\alpha \widehat{u}_2, \mu_2)_F \nonumber \\
& - \sum_{F \in \mathcal{F}_{h,\Gamma,1}} (\rho_1, \widehat{u}_1)_F 
- \sum_{F \in \mathcal{F}_{h,\Gamma,2}} (\rho_2, \widehat{u}_2)_F \nonumber \\
&+ \sum_{F \in \mathcal{F}_{h,\Gamma,1}} (\rho_1, \lambda)_F
+ \sum_{F \in \mathcal{F}_{h,\Gamma,2}} (\rho_2, \kappa \lambda)_F \nonumber \\
& + \sum_{F \in \mathcal{F}_{h,\Gamma,1}} (\mathcal{J}_1, \mu_1)_F 
+ \sum_{F \in \mathcal{F}_{h,\Gamma,2}} (\mathcal{J}_2, \mu_2)_F \nonumber \\
&- \sum_{F \in \mathcal{F}_{h,\Gamma}} (\mathcal{J}_1, \varphi)_F 
- \sum_{F \in \mathcal{F}_{h,\Gamma}} (\mathcal{J}_2, \varphi)_F,
\label{eq:bilinear_C_dmh}
\end{align}
\end{subequations}
and the following linear functionals:
\begin{subequations}\label{eq:linear_forms_dmh}
\begin{align}
& F({\tt v}): = 
(g, \phi_1)_1 + (g, \phi_2)_2, \label{eq:functional_F_dmh} \\
& G({\tt q}): = 
-\sum_{F \in \mathcal{F}_{h,\Sigma_1}} (\beta, \mu_1)_F 
-\sum_{F \in \mathcal{F}_{h,\Sigma_2}} (\beta, \mu_2)_F
-\sum_{F \in \mathcal{F}_{h,\Gamma}} (\sigma, \varphi)_F.
\label{eq:functional_G_dmh} 
\end{align}
\end{subequations}

\noindent Finally, the DMH weak formulation of problem~\eqref{eq:tx_model} can be written in abstract form as stated below.

\begin{definition}[DMH weak formulation]\label{def:DMH_wf}
\begin{subequations}\label{eq:saddle_point_problem_dmh}
Given the linear functionals $F: \mathcal{V}\rightarrow \mathbb R$ and $G: \mathcal{Q}\rightarrow\mathbb R$
defined in~\eqref{eq:linear_forms_dmh}, 
find ${\tt u} = (\mathbf{J}_1, u_1, \mathbf{J}_2, u_2) \in \mathcal{V}$ and 
${\tt p} = (\widehat{u}_1, \widehat{u}_2, \mathcal{J}_1, \mathcal{J}_2, \lambda) \in \mathcal{Q}$ such that:
\begin{align}
& a({\tt u}, {\tt v})  + b({\tt v}, {\tt p}) = F({\tt v}) \qquad \forall 
{\tt v} = (\boldsymbol{\tau}_1, \phi_1, \boldsymbol{\tau}_2, \phi_2) \in \mathcal{V}, \\
& b({\tt u}, {\tt q})  - c({\tt p}, {\tt q}) = G({\tt q}) \qquad \forall {\tt q} 
= (\mu_1, \mu_2, \rho_1, \rho_2, \varphi) \in \mathcal{Q},
\end{align}
where $\mathcal{V}:= \mathbf{V}_1 \times V_1 \times \mathbf{V}_2 \times V_2$, 
$\mathcal{Q}:= M_1 \times M_2 \times M_{J,1} \times M_{J,1} \times M_{\lambda}$ and 
the bilinear forms $a$, $b$ and $c$ are defined in~\eqref{eq:bilinear_forms_dmh}. 
\end{subequations}
\end{definition}

\begin{remark}
System~\eqref{eq:saddle_point_problem_dmh} is an instance of 
abstract generalized saddle-point problems~\eqref{eq:abstract_SP_problem}.
We notice that $a(\cdot, \cdot)$ and $b(\cdot, \cdot)$ are the standard bilinear forms in a dual mixed 
hybrid formulation
of a second-order boundary value problem with an advection-diffusion-reaction operator
(see~\cite{douglas_roberts_1985,arbogast_and_chen_1995}). On the contrary, the bilinear form 
$c(\cdot, \cdot)$ and the right-hand side $G(\cdot)$ contain the contributions of the 
Lagrange multipliers, conceptually borrowed from the 3F formulation, which allow us to enforce the 
transmission conditions~\eqref{eq:tx_flux} and~\eqref{eq:tx_segr}. These contributions represent 
a novel aspect of the DMH method proposed in this article.
\end{remark}

\begin{remark}
Using the fact that functions
$\mu_i \in M_i$ are single-valued on each face of $\mathcal{F}_{h,i}$, $i=1,2$, we see 
that the bilinear form $b({\tt u}, {\tt q})$ defined in~\eqref{eq:bilinear_B_dmh} can be written
in the following alternative (equivalent) manner
\begin{align}
& b({\tt u}, {\tt q}): = 
\sum_{i=1}^2 \sum_{K \in \mathcal{T}_{h,i}} 
(\mathbf{J}_i \cdot \mathbf{n}_{\partial K}, \mu_i)_{\partial K}
= \sum_{i=1}^2 \left( \sum_{F \in \mathcal{F}_{h,int,i}} (\jmp{\mathbf{J}_i}_F, \mu_i)_F \right. & \nonumber \\
& \left. + \sum_{F \in \mathcal{F}_{h,\Sigma_i}} (\mathbf{J}_i \cdot \mathbf{n}_{\partial K}, \mu_i)_F
+ \sum_{F \in \mathcal{F}_{h,\Gamma,i}} (\mathbf{J}_i \cdot \mathbf{n}_{\partial K}, \mu_i)_F
\right), &
\label{eq:b_form_faces}
\end{align}
where 
\begin{align}
& \jmp{\mathbf{J}_i}_F:= \mathbf{J}_i^+|_F \cdot \mathbf{n}_{F}^+ 
+ \mathbf{J}_i^-|_F \cdot \mathbf{n}_{F}^- & \label{eq:jump}
\end{align}
is the jump of $\mathbf{J}_i$ across the face $F \in \mathcal{F}_{h,int,i}$, $i=1,2$. 
In Eq.~\eqref{eq:jump}, $\mathbf{J}_i^+|_F$ and $\mathbf{J}_i^-|_F$ denote 
the trace on $F$ of the restrictions of 
$\mathbf{J}_i$ to the pair of elements $K^+$ and $K^-$ belonging to $\mathcal{T}_{h,i}$ such that 
$F = \partial K^+ \cap \partial K^-$, whereas $\mathbf{n}_{F}^+$ and $\mathbf{n}_{F}^-$
are the restrictions to $F$ of the outward unit normal vectors 
$\mathbf{n}_{\partial K}^+$ and $\mathbf{n}_{\partial K}^-$, respectively, with 
$\mathbf{n}_{F}^+ + \mathbf{n}_{F}^- = \mathbf{0}$.
\end{remark}

\subsection{Well-posedness of the DMH weak formulation}\label{sec:well-posedness_dmh}

The well-posedeness of problem~\eqref{eq:saddle_point_problem_dmh} 
is the main result of this section and is a lemma of the abstract Theorem A.1 reported in~\ref{sec:saddle_point_continuous}.
\begin{lemma}[Well posedness of~\eqref{eq:saddle_point_problem_dmh}]
\label{theorem:wellposedness_dmh_weak}
Assume that $r \in L^{\infty}(\Omega)$ with 
\begin{align}
& 0 < r_{min} \leq r(\mathbf{x}) \leq r_{max} < +\infty
& \textrm{for a.e.} \,\mathbf{x} \in \Omega. & \label{eq:r_bounds}
\end{align}
Assume that $g \in L^2(\Omega)$, $\mathbf{v} \in (L^\infty(\Omega))^3$ and $\alpha \in L^\infty(\Sigma)$
with $\alpha > 0$ a.e. on $\Sigma$. Assume also that
\begin{align}
& \| \mathbf{v} \|_{\infty, \Omega} < 2 \mu_{min} c_0 & \textrm{for a.e.} \,\mathbf{x} 
\in \Omega, & \label{eq:v_bounds}
\end{align} 
where $c_0:= \min \left\{ \mu_{max}^{-1}, r_{min} \right\}$, and that
\begin{align}
& \max \left( \kappa, \| \alpha \|_{\infty, \Sigma} \right) < \Frac{1}{C^\ast \mathcal{M}}, & 
\label{eq:kappa_bounds}
\end{align} 
where $C^\ast$ is the smallest trace constant over $\mathcal{T}_h$ and 
\begin{align}
& \mathcal{M}:= \left( 2+ r_{max} + \Frac{1+ \| \mathbf{v} \|_{\infty,\Omega}}{\mu_{min}} \right)
\left( \Frac{c_0 + 2 + r_{max} + \Frac{1}{\mu_{min}} + \Frac{\| \mathbf{v} \|_{\infty,\Omega}}{2 \mu_{min}}}
{c_0 - \Frac{\| \mathbf{v} \|_{\infty,\Omega}}{2 \mu_{min}}}\right). & \label{eq:Mcal}
\end{align}
Then, the DMH weak formulation~\eqref{eq:saddle_point_problem_dmh} of 
problem~\eqref{eq:tx_model} has a unique solution.
\end{lemma}

\begin{proof}
We apply Theorem A.1 reported in~\ref{sec:saddle_point_continuous}.
The first step of the proof is the verification of Assumptions~\eqref{eq:assumptions_continuity_forms}.
Using the discrete Cauchy-Schwarz inequality, the bound~\eqref{eq:mu_bound} and 
definitions~\eqref{eq:norm_spaces_dmh}, we see that~\eqref{eq:assumptions_continuity_forms} are
satisfied by taking:
\begin{subequations}\label{eq:constants_dmh}
\begin{align}
& M_a = \Frac{1 + \| \mathbf{v} \|_{\infty, \Omega}}{\mu_{min}}
+  \| r \|_{\infty, \Omega} + 2, & \label{eq:M_a_DMH} \\
& M_b = 1, & \label{eq:M_b_DMH} \\
& M_c = \max\left\{ \kappa, \| \alpha \|_{\infty,\Sigma}\right\}. & \label{eq:M_c_DMH}
\end{align}
\end{subequations}

The second step of the proof is the verification of~\eqref{eq:weak_coercivity_a}. Building upon the analysis
of~\cite[Section IV.1.4]{brezzifortin1991}, we see that
\begin{align*}
& \mathcal{V}^0=\left\{ (\mathbf{q}_1, \mathbf{q}_2) \in (\mathcal{H}({\rm div};\Omega_1) \times 
\mathcal{H}({\rm div};\Omega_2)), \right. \nonumber \\
& \left. \qquad {\rm with} \, \, {\rm div}\mathbf{q}_i = 0,  \, \, {\rm and} \, \, 
\mathbf{q}_i \cdot \mathbf{n}|_{\partial \Omega_i} = 0, \, i=1,2 \right\}.
\end{align*}
Then, for all ${\tt u} \in \mathcal{V}^0$, using Young's inequality we obtain
\begin{align*}
& a({\tt u}, {\tt u}) = (\mu^{-1} \mathbf{J}_1, \mathbf{J}_1)_1 +  (\mu^{-1} \mathbf{J}_2, \mathbf{J}_2)_2
+ (r u_1, u_1)_1 + (r u_2, u_2)_2 \nonumber \\
& - (\mu^{-1} \mathbf{v} u_1, \mathbf{J}_1)_1 - (\mu^{-1} \mathbf{v} u_2, \mathbf{J}_2)_2 \nonumber \\
& \geq \mu_{max}^{-1} \left[ \| \mathbf{J}_1 \|_{0,\Omega_1}^2 + \| \mathbf{J}_2 \|_{0,\Omega_2}^2 \right]
+ r_{min} \left[ \| u_1 \|_{0,\Omega_1}^2 + \| u_2 \|_{0,\Omega_2}^2 \right] \nonumber \\
& - \Frac{\| \mathbf{v} \|_{\infty, \Omega}}{2 \mu_{min}} 
\left[ \| u_1 \|_{0,\Omega_1}^2 + \| u_2 \|_{0,\Omega_2}^2 +
\| \mathbf{J}_1 \|_{0,\Omega_1}^2 + \| \mathbf{J}_2 \|_{0,\Omega_2}^2 \right] \nonumber \\
& \geq c_0 \left[ \| u_1 \|_{0,\Omega_1}^2 + \| u_2 \|_{0,\Omega_2}^2 \right] +
c_1 \left[ \| \mathbf{J}_1 \|_{0,\Omega_1}^2 + \| \mathbf{J}_2 \|_{0,\Omega_2}^2 \right]
\end{align*}
having set
\begin{align*}
& c_0:= r_{min} - \Frac{\| \mathbf{v} \|_{\infty, \Omega}}{2 \mu_{min}}, \qquad
c_1:= \mu_{max}^{-1} - \Frac{\| \mathbf{v} \|_{\infty, \Omega}}{2 \mu_{min}}.
\end{align*}
If assumption~\eqref{eq:v_bounds} holds, then $c_0>0$ and $c_1>0$ 
and~\eqref{eq:weak_coercivity_a} is satisfied by taking $k_a= \min \left\{ c_0, c_1 \right\}$.

The third step of the proof is the verification of~\eqref{eq:weak_coercivity_b}. To this end, we
set ${\tt U}:= (\mathbf{J}_1, 0, \mathbf{J}_2, 0) \in \mathcal{V}$, and 
for any given ${\tt P}:= (\mu_1, \mu_2, 0, 0, 0) \in \mathcal{Q}$ 
we consider the following auxiliary boundary value problems:
\begin{subequations}\label{eq:auxiliary_bvp_dmh}
\begin{align}
& - {\rm div}\mathbf{J}_i + w_i = 0 & {\rm in} \, K \in \mathcal{T}_{h,i} \quad i=1,2, 
\label{eq:equation_auxiliary_dmh} \\
& \mathbf{J}_i = \nabla w_i & {\rm in} \, K \in \mathcal{T}_{h,i} \quad i=1,2, 
\label{eq:equation_flux_dmh} \\
& w_i|_{\partial K} = \mu_i & {\rm on} \, \partial K \quad i=1,2. \label{eq:Dirichlet_BC_auxiliary_dmh}
\end{align}
\end{subequations}
The application of the dual mixed method to~\eqref{eq:auxiliary_bvp_dmh} 
and the use of Green's formula leads to the following localized saddle-point problem:
\begin{subequations}\label{eq:aux_mixed_dmh}
\begin{align}
& (\mathbf{J}_i, \mathbf{q}_i)_K + (w_i, {\rm div}\mathbf{q}_i)_K = (\mu_i, \mathbf{q}_i
\cdot \mathbf{n}_{\partial K})_{\partial K} & \mathbf{q}_i \in \mathcal{H}({\rm div}; K), 
\label{eq:mixed_1} \\
& (-{\rm div}\mathbf{J}_i + w_i, \phi_i)_K = 0     & \phi_i \in L^2(K). \label{eq:mixed_2}
\end{align}
\end{subequations}
Taking $\mathbf{q}_i = \mathbf{J}_i$ in~\eqref{eq:mixed_1} and $\phi_i = {\rm div}\mathbf{J}_i$ in~\eqref{eq:mixed_2}
we obtain
\begin{align}
& (\mu_i, \mathbf{J}_i \cdot \mathbf{n}_{\partial K})_{\partial K}
= \| \mathbf{J}_i \|_{H({\rm div}; K)}^2 \qquad \forall K \in \mathcal{T}_{h,i}, \, \, i=1,2. &
\label{eq:norm_H_minus_one_half}
\end{align}
Summing over the elements and over the two domains $\Omega_i$, $i=1,2$, and 
using~\eqref{eq:norm_H_minus_one_half}, we obtain:
\begin{align}
& b({\tt U}, {\tt P}) = \sum_{K \in \mathcal{T}_{h,1}} 
(\mathbf{J}_1 \cdot \mathbf{n}_{\partial K}, \mu_1)_{\partial K} + 
\sum_{K \in \mathcal{T}_{h,2}} (\mathbf{J}_2 \cdot \mathbf{n}_{\partial K}, 
\mu_2)_{\partial K} \nonumber \\
& = \sum_{K \in \mathcal{T}_{h,1}} \| \mathbf{J}_1 \|^2_{H({\rm div}; K)} + 
\sum_{K \in \mathcal{T}_{h,2}} \| \mathbf{J}_2 \|^2_{H({\rm div}; K)} 
\equiv \| {\tt U} \|_{\mathcal{V}}^2.
\label{eq:calcolo_b_dmh}
\end{align}
Using~\eqref{eq:equation_flux_dmh} and the definition of norm in $H({\rm div};K)$, 
we obtain the following identity for all $w_i \in H^1(K)$, $i=1,2$
\begin{align}
& \| \mathbf{J}_i \|^2_{H({\rm div}; K)} = \| \nabla w_i \|_{0,K}^2 + \| w_i \|_{0,K}^2
= \| w_i \|_{H^1(K)}^2. & \label{eq:norm_H1_K}
\end{align}
Substituting~\eqref{eq:norm_H1_K} into~\eqref{eq:calcolo_b_dmh} and applying
the trace inequality over each element $K$ yields
$$
\| {\tt U} \|_{\mathcal{V}} \geq  (C^\ast)^{-1} 
\left( \sum_{K \in \mathcal{T}_{h,1}} \| \mu_1 \|_{0, \partial K}^2
+ \sum_{K \in \mathcal{T}_{h,2}} \| \mu_2 \|_{0,\partial K}^2 \right)^{1/2},
$$
where $C^\ast$ is the smallest trace constant over $\mathcal{T}_h$.
Replacing the above relation into~\eqref{eq:calcolo_b_dmh}, we 
see that~\eqref{eq:weak_coercivity_b} is satisfied by taking $k_b=(C^\ast)^{-1}$.

The fourth step of the proof is the verification of~\eqref{eq:smallness_constants}. 
Inserting the values of the continuity constants $M_a$ and $M_c$ and of the coercivity constants
$k_a$ and $k_b$ found above into the definition of $\delta$ in~\eqref{eq:smallness_constants}, we see that
if~\eqref{eq:kappa_bounds} holds then $\delta < 1$ and~\eqref{eq:smallness_constants} is satisfied.
This completes the proof.
\end{proof}

\begin{remark}\label{rem:equivalence_saddle_point_abstract}
Having proved that~\eqref{eq:saddle_point_problem_dmh} is uniquely solvable, we see 
that~\eqref{eq:saddle_point_problem_dmh} can be written in the equivalent form: 
given $F \in \mathcal{V}^{\prime}$, find ${\tt u} \in \mathcal{V}^{\mathcal{H}}$ such that
\begin{subequations}\label{eq:saddle_equivalent}
\begin{align}
& a({\tt u}, {\tt v}) = F({\tt v}) \qquad \forall {\tt v} 
\in \mathcal{V}^0, & \label{eq:saddle_point_equivalent}
\end{align}
where, for ${\tt p}$ given in $\mathcal{Q}$, we set
\begin{align}
& \mathcal{H}({\tt q}):= G({\tt q}) + c({\tt p},{\tt q}) 
\qquad \forall {\tt q} \in \mathcal{Q}, & \label{eq:new_functional_H}
\end{align}
and $\mathcal{V}^{\mathcal{H}}$ is the affine manifold
\begin{align}
& \mathcal{V}^{\mathcal{H}} := \left\{ {\tt v} \in \mathcal{V}, \, \, 
b({\tt v}, {\tt q})= \mathcal{H}({\tt q})  \, \, \forall {\tt q} \in \mathcal{Q} \right\}. \label{eq:bspace}
\end{align}
\end{subequations}
\end{remark}

\section{The DMH Galerkin finite element approximation}\label{sec:dmh_Galerkin_FE_approximation}
In this section we illustrate the Galerkin finite element approximation of 
the weak DMH formulation of problem~\eqref{eq:tx_model}.
To this end, in Section~\ref{sec:fem_spaces} we introduce  the local and global finite element 
spaces. Then, in Section~\ref{sec:DMH_FE_method} we use the abstract theory 
reported in~\ref{sec:saddle_point_approx}
to prove that the DMH formulation admits a unique solution 
and exhibits optimal convergence with respect to the discretization parameter $h$. 

\subsection{Finite element spaces}\label{sec:fem_spaces}
For any set $\mathcal{S}$ (in one, two or three spatial dimensions), 
we indicate by $\mathbb{P}_r(\mathcal{S})$, $r \geq 0$,
the space of polynomials of degree $\leq r$ defined on $\mathcal{S}$.
Moreover, we define $RT_0(K) := 
(\mathbb{P}_0(K))^3 \oplus \mathbb{P}_0(K) \mathbf{x}$
and we define the following local polynomial spaces associated with the triangulation $\mathcal{T}_h$:
\begin{subequations}\label{eq:local_spaces}
\begin{align}
& \mathbf{V}(K):=\left\{ \mathbf{v} \in RT_0(K)   \, \, 
\forall K \in \mathcal{T}_h \right\}, & \label{eq:Vvect_K} \\
& V(K):=\left\{ v \in \mathbb{P}_0(K) \, \, 
\forall K \in \mathcal{T}_h \right\}, & \label{eq:V_K} \\
& M(F):=\left\{ \mu \in \mathbb{P}_0(F) \, \, 
\forall F \in \mathcal{F}_h \right\}. & \label{eq:M_F}
\end{align}
\end{subequations}
\begin{figure}[h!]
\centering
\includegraphics[width=0.45\textwidth]{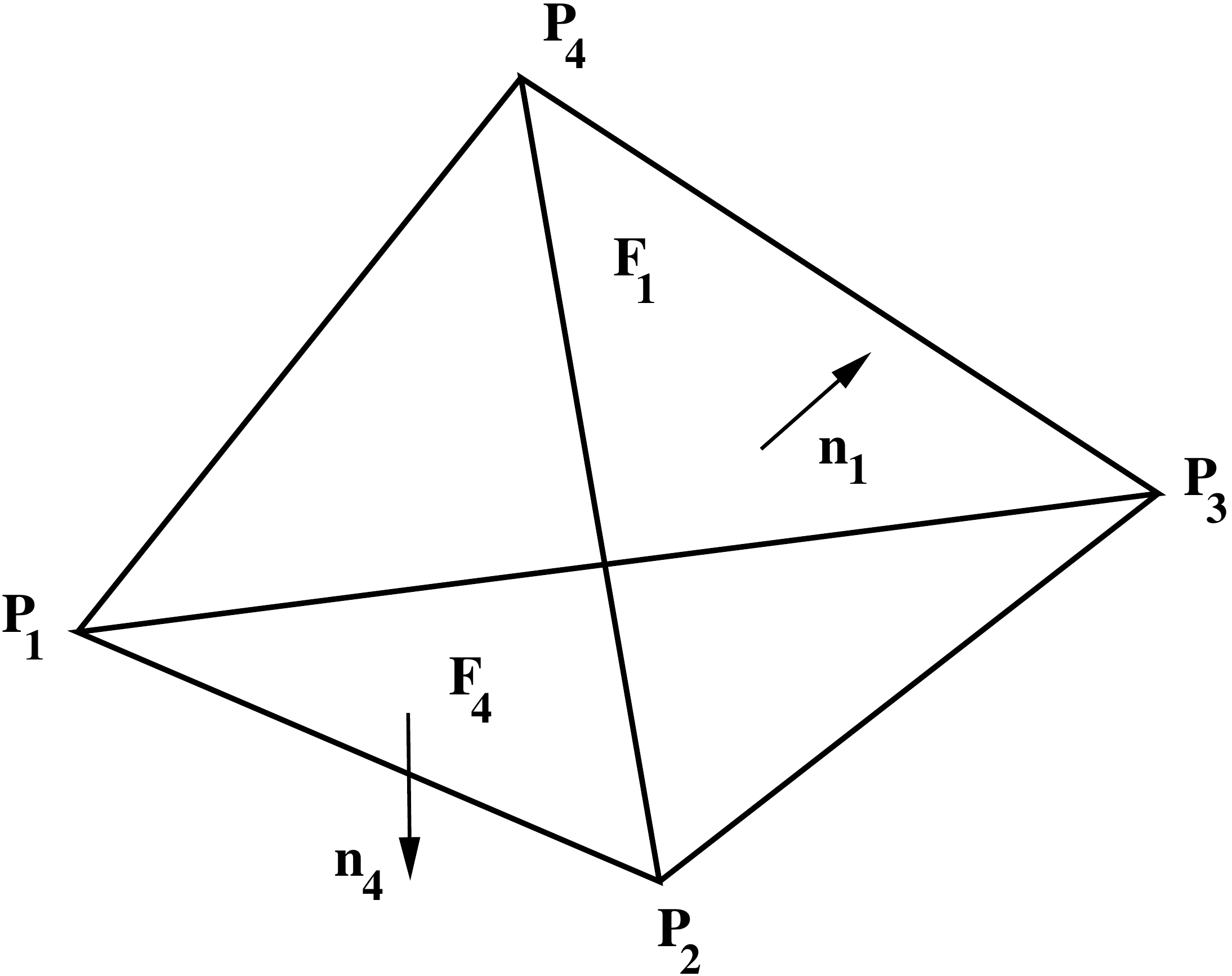}
\caption{The element $K$ and the local geometrical notation.}
\label{fig:tet_K}
\end{figure}

The above spaces are local because they are defined within each element $K$ of the triangulation 
$\mathcal{T}_h$ and for each face $F$ of the set of faces $\mathcal{F}_h$.
We remark that:
\begin{subequations}\label{eq:dim_local_spaces}
\begin{align}
& {\rm dim}(\mathbf{V}(K)) = 4, & \label{eq:dim_Vvect_K} \\
& {\rm dim}(V(K))= {\rm dim}(M(F)) = 1. & \label{eq:dim_V_K_M_F}
\end{align}
\end{subequations}
The degrees of freedom for a vector-valued function $\mathbf{v} \in \mathbf{V}(K)$ are the
fluxes of $\mathbf{v}$ across each face of the boundary $\partial K$
\begin{align}
& \Phi^K_i(\mathbf{v}):= \int_{F_i \in \partial K} \mathbf{v} \cdot \mathbf{n}_i \, dF & i=1,2,3,4, 
\label{eq:dofs_Vvect_K}
\end{align}
where face $F_i$ is opposite to vertex $P_i$, $i=1,2,3,4$, and $\mathbf{n}_i$ is the outward unit
normal vector on $F_i$ (see Figure~\ref{fig:tet_K}). 

\noindent
Using definition~\eqref{eq:dofs_Vvect_K}
we can write the generic function $\mathbf{v} \in \mathbf{V}(K)$ as
\begin{align}
& \mathbf{v}(\mathbf{x}) = \sum_{i=1}^4 \Phi^K_i(\mathbf{v}) \boldsymbol{\tau}_i(\mathbf{x}), 
& \label{eq:expansion_v}
\end{align}
where 
\begin{align}
& \boldsymbol{\tau}_i(\mathbf{x}):= \Frac{\mathbf{x} - \mathbf{x}_i}{3 |K|} & i=1,2,3,4, 
\label{eq:RT0_local_basis} 
\end{align}
are the local basis functions of the Raviart-Thomas/Nedelec 
finite element space of lowest order (RT0,
see~\cite{RaviartThomas1977,Nedelec1980,RobertsThomas1991}), 
having denoted by $\mathbf{x}_i$ and $|K|$ the coordinates of vertex $P_i$ and
the volume of $K$, respectively.
The degree of freedom of a function $v$ belonging to $V(K)$ is the value of $v$ at the barycenter
of $K$, whereas the degree of freedom of a function $\mu$ belonging to $M(F)$ is the value of $\mu$ 
at the barycenter of $F$.

\noindent
In order to construct the finite dimensional spaces associated with~\eqref{eq:local_spaces} to be used
for the internal approximation of the functional spaces~\eqref{eq:dmh_spaces},
we distinguish between the spaces of functions defined inside each element of $\mathcal{T}_h$
and the spaces of functions defined on each face of $\mathcal{F}_h$. For $i=1,2$, we have:
\begin{subequations}\label{eq:global_spaces_Tau_h_i}
\begin{align}
& \mathbf{V}_{i,h} = \left\{ \boldsymbol{\tau} \in \mathbf{V}_{i}, \, 
\boldsymbol{\tau}|_K \in \mathbf{V}(K) \, \, \forall K \in \mathcal{T}_{h,i} \right\}, 
\, i=1,2, \label{eq:spaceVvect_1_2} \\
& V_{i,h} = \left\{ \phi \in V_{i}, \, \phi|_K \in V(K) \, \, \forall K 
\in \mathcal{T}_{h,i} \right\}, \, i=1,2, \label{eq:spaceV_1_2} \\
& M_{i,h} = \left\{ \mu \in  M_i, \, \mu|_F \in M(F) \, \, \forall F \in \mathcal{F}_{h,i}\right\}, \, i=1,2,
 \label{eq:spaceM_i_h} \\ 
& M_{J,i,h} = \left\{ \mu \in M_{J,i}, \, \mu|_F \in M(F) \, \, \forall F \in \mathcal{F}_{h,\Gamma,i}\right\}, 
\, i=1,2, \label{eq:spaceM_J_i_h} \\
& M_{\lambda,h} = \left\{ \mu \in M_{\Lambda}, \, \mu|_F \in M(F) \, \, \forall F \in \mathcal{F}_{h,\Gamma}
\right\}. \label{eq:M_space_lambda_dmh_h}
\end{align}
\end{subequations}
Having defined the global finite element spaces on the partitioned triangulation, we can define
the global spaces on $\mathcal{T}_h$ as:
\begin{subequations}\label{eq:global_spaces_Tau_h}
\begin{align}
& \mathcal{V}_h:= \mathbf{V}_{1,h} \times V_{1,h} \times \mathbf{V}_{2,h} \times V_{2,h},
\label{eq:space_cal_V_h} \\
& \mathcal{Q}_h:= M_{1,h} \times M_{2,h} \times M_{J,1,h} \times M_{J,2,h} \times M_{\lambda,h}.
\label{eq:space_cal_Q_h}
\end{align}
\end{subequations}

\subsection{The DMH numerical method}\label{sec:DMH_FE_method}

The DMH-RT0 FEM approximation 
of problem~\eqref{eq:tx_model} can be written in abstract form as:
\begin{subequations}\label{eq:saddle_point_problem_dmh_h}
given $F_h=F({\tt v}_h)$ and $G_h=G({\tt q}_h)$, 
find ${\tt u}_h = (\mathbf{J}_{1,h}, u_{1,h}, \mathbf{J}_{2,h}, u_{2,h}) \in \mathcal{V}_h$ 
and ${\tt p}_h = (\widehat{u}_{1,h}, \widehat{u}_{2,h}, \mathcal{J}_{1,h}, \mathcal{J}_{2,h}, 
\lambda_h) \in \mathcal{Q}_h$ such that:
\begin{align}
& a({\tt u}_h, {\tt v}_h)  + b({\tt v}_h, {\tt p}_h) = F_h \qquad \forall {\tt v}_h 
= (\boldsymbol{\tau}_{1,h}, \phi_{1,h}, \boldsymbol{\tau}_{2,h}, 
\phi_{2,h}) \in \mathcal{V}_h, \\
& b({\tt u}_h, {\tt q}_h)  - c({\tt p}_h, {\tt q}_h) = G_h \qquad \forall {\tt q}_h 
= (\mu_{1,h}, \mu_{2,h}, \rho_{1,h}, \rho_{2,h}, \varphi_h) \in \mathcal{Q}_h,
\end{align}
where the bilinear forms $a(\cdot, \cdot)$, $b(\cdot, \cdot)$, $c(\cdot, \cdot)$
are defined in~\eqref{eq:bilinear_forms_dmh}, the linear functionals $F(\cdot)$, $G(\cdot)$
are defined in~\eqref{eq:linear_forms_dmh} and the spaces $\mathcal{V}_h$ and $\mathcal{Q}_h$ are defined in~\eqref{eq:global_spaces_Tau_h}. 
\end{subequations}

System~\eqref{eq:saddle_point_problem_dmh_h} is a special instance of  
the approximate generalized saddle-point problem~\eqref{eq:abstract_SP_problem_h}.
Since $\mathcal{V}_h \subset \mathcal{V}$ and $\mathcal{Q}_h \subset \mathcal{Q}$, the discrete 
continuity constants of the bilinear forms~\eqref{eq:bilinear_forms_dmh} are
$M_{a,h} = M_a$, $M_{b,h} = M_b$ and $M_{c,h} = M_c$. Similarly, the discrete 
coercivity constants are $k_{a,h}=k_a$ and $k_{b,h}=k_b$. Moreover, since 
${\rm div}\mathbf{V}(K) = V(K)$ for all $K \in \mathcal{T}_h$, we have 
\begin{align*}
& \mathcal{V}^0_h=\left\{ (\mathbf{q}_1, \mathbf{q}_2) \in 
(\mathbf{V}_{1,h} \times \mathbf{V}_{2,h}) \cap (\mathcal{H}({\rm div};\Omega_1) \times 
\mathcal{H}({\rm div};\Omega_2)), \right. \nonumber \\
& \left. \qquad {\rm with} \, \, {\rm div}\mathbf{q}_i = 0,  \, \, {\rm and} \, \, 
\mathbf{q}_i \cdot \mathbf{n}|_{\partial \Omega_i} = 0, \, i=1,2 \right\} \subset \mathcal{V}^0.
\end{align*}
The following lemma is a result of Theorem A.2 reported in~\ref{sec:saddle_point_approx}.
\begin{lemma}[Well posedness of~\eqref{eq:saddle_point_problem_dmh_h}]
\label{theorem:wellposedness_dmh_weak_h} 
The DMH-RT0 FEM approximation~\eqref{eq:saddle_point_problem_dmh_h} of problem~\eqref{eq:tx_model} has a unique solution.
\end{lemma}

\begin{remark}\label{rem:equivalence_saddle_point_abstract_h}
In analogy with Remark~\ref{rem:equivalence_saddle_point_abstract}, 
having proved that~\eqref{eq:saddle_point_problem_dmh_h} is uniquely solvable, we see 
that~\eqref{eq:saddle_point_problem_dmh_h} can be written in the equivalent form: 
given $F_h$, find ${\tt u}_h \in \mathcal{V}^{\mathcal{H}}_h$ such that
\begin{subequations}\label{eq:saddle_equivalent_h}
\begin{align}
& a({\tt u}_h, {\tt v}_h) = F_h \qquad \forall 
{\tt v}_h \in \mathcal{V}^0_h, & \label{eq:saddle_point_equivalent_h}
\end{align}
where, for ${\tt p}_h$ given in $\mathcal{Q}_h$, we set
\begin{align}
& \mathcal{H}_h = \mathcal{H}({\tt q}_h):= G({\tt q}_h) + 
c({\tt p}_h,{\tt q}_h) \qquad \forall {\tt q}_h \in \mathcal{Q}_h. & \label{eq:new_functional_H_h}
\end{align}
\end{subequations}
\end{remark}

The equivalence results discussed in Remarks~\ref{rem:equivalence_saddle_point_abstract} and~\ref{rem:equivalence_saddle_point_abstract_h} allow us to apply Theorem 7.4.3 of~\cite{quarteroni1994numerical}, which proves the convergence of the solution of~\eqref{eq:saddle_point_problem_dmh_h} 
to the solution of~\eqref{eq:saddle_point_problem_dmh}. 
\begin{theorem}[Convergence of the Galerkin approximation]\label{theo:abstract_error_analysis_h}
Since all hypotheses of Theorem A.1 and Theorem A.2 
are satisfied, the following optimal error estimates hold:
\begin{subequations}\label{eq:theorem_approx_h}
\begin{align}
& \| {\tt u} - {\tt u}_h \|_{\mathcal{V}} \leq E_{11,h} \inf_{{\tt v}_h \in \mathcal{V}_h} 
\| {\tt u} - {\tt v}_h \|_{\mathcal{V}}
+ E_{12,h} \inf_{{\tt q}_h \in \mathcal{Q}_h} \| {\tt p} - {\tt q}_h \|_{\mathcal{Q}} & \label{eq:err_estimate_u_h} \\
& \| {\tt p} - {\tt p}_h \|_{\mathcal{Q}} \leq E_{21,h} 
\inf_{{\tt v}_h \in \mathcal{V}_h} \| {\tt u} - {\tt v}_h \|_{\mathcal{V}}
+ E_{22,h} \inf_{{\tt q}_h \in \mathcal{Q}_h} \| {\tt p} - {\tt q}_h \|_{\mathcal{Q}}, & \label{eq:err_estimate_p_h}
\end{align}
\end{subequations}
where:
\begin{align*}
& E_{11,h} = \left( 1 + \Frac{M_a}{k_{a,h}} \right) \left( 1 + \Frac{M_b}{k_{b,h}} \right), 
\qquad E_{12,h} = \Frac{M_b}{k_{b,h}}, \\
& E_{21,h} = \Frac{M_a}{k_{b,h}} E_{11,h}, \qquad 
E_{22,h} =  \left( 1 + \Frac{M_b}{k_{b,h}} + \Frac{M_a \, M_b}{k_{a,h} \, k_{b,h}} \right). 
\end{align*}
\end{theorem}

Using in~\eqref{eq:theorem_approx_h} the approximation theory for hybrid methods developed 
in~\cite{RobertsThomas1991} yields the following convergence estimates for the DMH-RT0 FEM.
\begin{theorem}[Convergence of the DMH-RT0 FEM]\label{theo:convergence_estimate_dmh}
There exist positive constants $C_{\tt u}$ and $C_{\tt p}$, independent of $h$, such that:
\begin{subequations}\label{eq:theorem_approx_h_dmh}
\begin{align}
& \| {\tt u} - {\tt u}_h \|_{\mathcal{V}} \leq C_{\tt u} \, h, & \label{eq:err_estimate_u_h_dmh} \\
& \| {\tt p} - {\tt p}_h \|_{\mathcal{Q}} \leq C_{\tt p} \, h. & \label{eq:err_estimate_p_h_dmh}
\end{align}
\end{subequations}
\end{theorem}
Moreover, using the techniques of~\cite{ArnoldBrezzi1985} and~\cite[Section~21]{RobertsThomas1991}, we can prove
the following post-processing error estimates.
\begin{theorem}[Convergence of post-processed quantities]\label{theo:convergence_estimate_post_dmh}
There exist positive constants $\widetilde{C}_1$ and $\widetilde{C}_2$, independent of $h$, such that:
\begin{subequations}\label{eq:theorem_approx_h_dmh_post}
\begin{align}
& \| P_0 u - u_h \|_{0,\Omega} \leq \widetilde{C}_1 \, h^2, & \label{eq:err_estimate_u_h_dmh_P_0u} \\
& \| u - u_h^\ast \|_{0,\Omega} \leq \widetilde{C}_2 \, h^2, & \label{eq:err_estimate_u_h_dmh_star}
\end{align}
where $P_0 u$ is the $L^2$ projection of $u$ on $V_{i,h}$ and $u_h^\ast$ is the 
piecewise linear nonconforming interpolant of $\widehat{u}_h$ over $\mathcal{T}_h$~\cite{CR1973}.
\end{subequations}
\end{theorem}

\begin{remark}
The error estimates~\eqref{eq:theorem_approx_h_dmh_post} are superconvergence results for the DMH-RT0
FEM. In particular, error estimate~\eqref{eq:err_estimate_u_h_dmh_P_0u} tells us that
$u_h$ is a very good approximation of $u$ at the barycenter of each element $K \in \mathcal{T}_h$,
whereas error estimate~\eqref{eq:err_estimate_u_h_dmh_star} tells us that the piecewise linear 
nonconforming interpolant
of $\widehat{u}_h$ over the mesh approximates the exact solution $u$ with the same accuracy 
as that of the piecewise linear solution computed by the standard finite element method applied 
to problem~\eqref{eq:tx_model}.
\end{remark}

\section{Efficient implementation of the DMH method}\label{sec:static_condensation}

In this section, we illustrate how to implement 
the DMH-RT0 FEM scheme \eqref{eq:saddle_point_problem_dmh_h} in a computationally efficient manner. 
To this end, 
we first discuss the properties of
the linear algebraic system and then describe in detail the static condensation procedure
that allows us to eliminate the internal variables $u_i, \mathbf{J}_i$ and 
the Lagrange multipliers $\mathcal{J}_i$ in favor of $\widehat{u}_i$ and $\lambda$, $i=1,2$.

\subsection{System reduction through static condensation}\label{sec:linear_system_DMH_FE_method}

Functions belonging to the finite dimensional space $\mathcal{V}_h$ are completely
discontinuous over $\mathcal{T}_h$. Similarly, functions belonging to
the finite dimensional space $\mathcal{Q}_h$ are completely discontinuous over $\mathcal{F}_h$.
These properties can be profitably exploited to implement the DMH-RT0 FEM scheme in a very efficient manner
through the use of static condensation. This procedure is basically 
a Gauss elimination algorithm that allows us to express all the variables of the numerical
method as a function of a sole unknown, thereby reducing considerably the size of the linear algebraic
system and enhancing the computational efficiency of the method. 
Static condensation, however, is not a feature specific of the DMH-RT0 FEM scheme
proposed in the present article, but is widely adopted in finite element formulations.
We refer to~\cite{ArnoldBrezzi1985,brezzifortin1991} for an introduction 
to static condensation in mixed and hybrid finite element methods, to~\cite{CockburnGopalakrishnanLazarov2009,Yakovlev2016} for 
an advanced use of static condensation in the context of Continuous and 
Hybridizable Discontinuous Galerkin methods and to~\cite{quarteroni1999domain} for a description of the use of 
static condensation as an algorithm to implement the method of Schur complement system.
To apply static condensation to the DMH FEM it is convenient to write  
the linear algebraic system associated with problem~\eqref{eq:saddle_point_problem_dmh_h} 
in full block form, which reads:
\begin{equation}\label{eq:dmh_FE_linear_system}
\left[
\begin{array}{lllllllll}
{\tt A}_1 & {\tt N}_1 & 0 & 0 & {\tt D}_1^T & 0 & 0 & 0 & 0 \\
{\tt P}_1 & {\tt R}_1 & 0 & 0 & 0 & 0 & 0 & 0 & 0 \\
0         &  0        & {\tt A}_2 & {\tt N}_2 & 0 & {\tt D}_2^T & 0 & 0 & 0 \\
0         &  0        & {\tt P}_2 & {\tt R}_2 & 0 & 0 & 0 & 0 & 0 \\
{\tt D}_1 & 0         & 0 & 0      & -{\tt M}_{\Sigma_1} & 0 & -{\tt E}_1^T & 0 & 0 \\
0         & 0         & {\tt D}_2 & 0   & 0 & -{\tt M}_{\Sigma_2} & 0 & -{\tt E}_2^T & 0  \\
0 & 0 & 0 & 0         & {\tt E}_1 & 0 & 0 & 0 & -{\tt U}_1^T \\
0 & 0 & 0 & 0         & 0       & {\tt E}_2 & 0 & 0 & -\kappa {\tt U}_2^T \\
0 & 0 & 0 & 0         & 0  & 0              & {\tt U}_1 & {\tt U}_2 & 0
\end{array}
\right]  
\left[
\begin{array}{l}
\mathbf{J}_1 \\
\mathbf{u}_1 \\
\mathbf{J}_2 \\
\mathbf{u}_2 \\
\widehat{\mathbf{u}}_1 \\
\widehat{\mathbf{u}}_2 \\
\mathbf{j}_1 \\
\mathbf{j}_2 \\
\boldsymbol{\lambda}
\end{array}
\right] 
= 
\left[
\begin{array}{l}
\mathbf{0} \\
\mathbf{b}_1 \\
\mathbf{0} \\
\mathbf{b}_2 \\
\mathbf{b}_{\Sigma_1} \\
\mathbf{b}_{\Sigma_2} \\
\mathbf{0} \\
\mathbf{0} \\
\mathbf{b}_{\sigma}
\end{array}
\right]. 
\end{equation}
In the equation system~\eqref{eq:dmh_FE_linear_system}, $\mathbf{J}_i$, $\mathbf{u}_i$, $i=1,2$,
denote the vectors of the degrees of freedom for the internal variables $\mathbf{J}_h$ and $u_h$
inside the partitioned triangulations $\mathcal{T}_{h,i}$, $i=1,2$. 
In particular, denoting by ${\tt NE}_1$ the number of tetrahedra in
$\mathcal{T}_{h,1}$ and by ${\tt NE}_2$ the number of tetrahedra in
$\mathcal{T}_{h,2}$, we notice that $\mathbf{J}_1$ is subdivided into a collection of 
${\tt NE}_1$ vectors of size equal to 4 and $\mathbf{u}_1$ has size equal to ${\tt NE}_1$; 
analogously,  $\mathbf{J}_2$ is subdivided into a collection of 
${\tt NE}_2$ vectors of size equal to 4 and $\mathbf{u}_2$ has size equal to ${\tt NE}_2$.
In the same spirit, matrix ${\tt A}_1$ has a block diagonal structure of size ${\tt NE}_1$,
where each block is the $4 \times 4$ flux matrix corresponding to an element of $\mathcal{T}_{h,1}$,
whereas matrix ${\tt A}_2$ has a block diagonal structure of size ${\tt NE}_2$,
where each block is the $4 \times 4$ flux matrix corresponding to an element of $\mathcal{T}_{h,2}$.
Similar considerations apply to the rectangular block matrices ${\tt P}_i$ and 
${\tt N}_i:= {\tt H}_i - {\tt P}_i^T$, $i=1,2$, and to the block matrices ${\tt R}_i$, $i=1,2$,
that have a diagonal structure, each entry corresponding to an element of $\mathcal{T}_{h,1}$
and $\mathcal{T}_{h,2}$, respectively.
The unknown vectors $\widehat{\mathbf{u}}_i$, instead, contain the degrees of freedom of the hybrid 
variables $\widehat{u}_{h,i}$, $i=1,2$, associated with each face of $\mathcal{F}_{h,i}$, $i=1,2$,
and for this reason the size of $\widehat{\mathbf{u}}_1$ is equal to ${\tt NF}_1$ and the size of
$\widehat{\mathbf{u}}_2$ is equal to ${\tt NF}_2$, where ${\tt NF}_1$ and ${\tt NF}_2$ denote
the number of faces in $\mathcal{F}_{h,1}$ and $\mathcal{F}_{h,2}$, the faces belonging to the interface
$\Gamma$ being counted twice. The unknown vectors $\mathbf{j}_i$, $i=1,2$, contain the degrees of freedom of 
the flux Lagrange multipliers $\mathcal{J}_{h,i}$, $i=1,2$, associated with each face of 
$\mathcal{F}_{h,\Gamma,1}$ and $\mathcal{F}_{h,\Gamma,2}$, respectively, and therefore 
their sizes are both equal to
${\tt NF}_\Gamma$, where ${\tt NF}_\Gamma$ denotes
the number of faces in $\mathcal{F}_{h,\Gamma}$. 
Finally, the unknown vector $\boldsymbol{\lambda}$ contains the degrees of freedom of 
the segregation condition Lagrange multiplier $\lambda_{h}$ associated with each face of 
$\mathcal{F}_{h,\Gamma}$, and therefore has size equal to ${\tt NF}_\Gamma$.
The matrices ${\tt D}_i$, $i=1,2$, enforce  the continuity of $\mathbf{J}_{h,i} \cdot \mathbf{n}_i$ 
across interelement boundaries in each triangulation $\mathcal{T}_{h,i}$.
The matrices ${\tt M}_{\Sigma_i}$, $i=1,2$, enforce the continuity of the Robin boundary 
boundary conditions~\eqref{eq:bcs} on each face of $\Sigma_i$.
The matrices ${\tt E}_i$, $i=1,2$, enforce the identity between 
$\mathbf{J}_{h,i} \cdot \mathbf{n}_i$ and the Lagrange multiplier $\mathcal{J}_{h,i}$
across each face belonging to $\mathcal{F}_{h,\Gamma,i}$, $i=1,2$.
%
The matrices ${\tt U}_1$ and ${\tt U}_2$ enforce the transmission 
condition~\eqref{eq:tx_flux} across each face of $\mathcal{F}_{h,\Gamma}$ whereas the matrices
$-{\tt U}_1^T$ and $-\kappa {\tt U}_2^T$ enforce the segregation condition~\eqref{eq:tx_segr}
across each face of $\mathcal{F}_{h,\Gamma}$. 
In analogy to what happens for the matrices 
associated with the internal degrees of freedom in each partitioned triangulation, 
also the matrices ${\tt D}_i$, ${\tt M}_{\Sigma_i}$, ${\tt E}_i$ and ${\tt U}_i$ have a block
structure, each block corresponding to a face of $\mathcal{F}_{h,i}$, $i=1,2$.
To conclude, the right-hand side vectors $\mathbf{b}_i$, $\mathbf{b}_{\Sigma_i}$ and
$\mathbf{b}_{\sigma}$, contain the contributions due to the source term $g$ 
in~\eqref{eq:continuity}, of the boundary terms $\beta_i$ in~\eqref{eq:bcs} and 
of the interface flux term $-\sigma$ in~\eqref{eq:tx_flux}, respectively.

\subsubsection{Elimination of the internal variables $\mathbf{J}_h$ and $u_h$}
The first and second equations in the block linear 
system~\eqref{eq:dmh_FE_linear_system} read:
\begin{subequations}\label{eq:dmh_FE_domain_Omega_1}
\begin{align}
& {\tt A}_1 \mathbf{J}_1 + {\tt N}_1 \mathbf{u}_1 + {\tt D}_1^T \widehat{\mathbf{u}}_1 = \mathbf{0} 
& \label{eq:dmh_omega_1_eq1} \\
& {\tt P}_1 \mathbf{J}_1 + {\tt R}_1 \mathbf{u}_1 = \mathbf{g}_1, 
& \label{eq:dmh_omega_1_eq2}
\end{align}
\end{subequations}
whereas the third and fourth equations in the block linear 
system~\eqref{eq:dmh_FE_linear_system} read:
\begin{subequations}\label{eq:dmh_FE_domain_Omega_2}
\begin{align}
& {\tt A}_2 \mathbf{J}_2 + {\tt N}_2 \mathbf{u}_2 + {\tt D}_2^T \widehat{\mathbf{u}}_2 = \mathbf{0} 
& \label{eq:dmh_omega_2_eq1} \\
& {\tt P}_2 \mathbf{J}_2 + {\tt R}_2 \mathbf{u}_2 = \mathbf{g}_2.
& \label{eq:dmh_omega_2_eq2}
\end{align}
\end{subequations}
The two systems~\eqref{eq:dmh_FE_domain_Omega_1} and~\eqref{eq:dmh_FE_domain_Omega_2}
have a \emph{local} nature, that is, the unknown vector pairs $(\mathbf{J}_i, \mathbf{u}_i)$, $i=1,2$,
are associated with each tetrahedron $K$ belonging to $\mathcal{T}_{h,1}$ and $\mathcal{T}_{h,2}$, 
respectively. In particular, we see that the 
$4 \times 4$ flux matrices ${\tt A}_i$, $i=1,2$, are symmetric and positive definite, 
so that~\eqref{eq:dmh_omega_1_eq1} and~\eqref{eq:dmh_omega_2_eq1} can be solved to obtain:
\begin{subequations}\label{eq:equations_for_J}
\begin{align}
& \mathbf{J}_1 = - {\tt A}_1^{-1} \left[{\tt N}_1 \mathbf{u}_1 + {\tt D}_1^T \widehat{\mathbf{u}}_1\right],
& \label{eq:dmh_FE_equation_for_J1} \\
& \mathbf{J}_2 = - {\tt A}_2^{-1} \left[{\tt N}_2 \mathbf{u}_2 + {\tt D}_2^T \widehat{\mathbf{u}}_2\right].
& \label{eq:dmh_FE_equation_for_J2}
\end{align}
\end{subequations}
Then, we can substitute the above expressions in~\eqref{eq:dmh_omega_1_eq2} and~\eqref{eq:dmh_omega_2_eq2}
to get:
\begin{subequations}\label{eq:equations_for_u}
\begin{align}
& - {\tt P}_1 {\tt A}_1^{-1} \left[{\tt N}_1 \mathbf{u}_1 + {\tt D}_1^T \widehat{\mathbf{u}}_1\right] 
+ {\tt R}_1 \mathbf{u}_1 = \mathbf{g}_1, & \label{eq:dmh_omega_1_eq2_bis} \\
& - {\tt P}_2 {\tt A}_2^{-1} \left[{\tt N}_2 \mathbf{u}_2 + {\tt D}_2^T \widehat{\mathbf{u}}_2\right] 
+ {\tt R}_2 \mathbf{u}_2 = \mathbf{g}_2. & \label{eq:dmh_omega_2_eq2_bis}
\end{align}
\end{subequations}
Letting 
$$
{\tt M}_i:= {\tt R}_i - {\tt P}_i {\tt A}_i^{-1} {\tt N}_i = 
{\tt R}_i + {\tt P}_i {\tt A}_i^{-1} {\tt P}_i^T - {\tt P}_i {\tt A}_i^{-1} {\tt H}_i, \qquad 
i=1,2, 
$$
equations~\eqref{eq:equations_for_u} become:
\begin{subequations}\label{eq:equations_for_u_bis}
\begin{align}
& {\tt M}_1 \mathbf{u}_1 - {\tt P}_1 {\tt A}_1^{-1} {\tt D}_1^T \widehat{\mathbf{u}}_1 = \mathbf{g}_1,& \label{eq:dmh_omega_1_eq2_bis_bis} \\
& {\tt M}_2 \mathbf{u}_2 - {\tt P}_2 {\tt A}_2^{-1} {\tt D}_2^T \widehat{\mathbf{u}}_2 = \mathbf{g}_2. & \label{eq:dmh_omega_2_eq2_bis_bis}
\end{align}
\end{subequations}
Matrices ${\tt M}_i$ have size $1 \times 1$ and are invertible because of assumption~\eqref{eq:weak_coercivity_a}.
Thus, equations~\eqref{eq:equations_for_u_bis} can be solved to obtain:
\begin{subequations}\label{eq:equations_for_U}
\begin{align}
& \mathbf{u}_1 = {\tt M}_1^{-1} \left[{\tt P}_1 {\tt A}_1^{-1} {\tt D}_1^T \widehat{\mathbf{u}}_1 
+ \mathbf{g}_1\right], & \label{eq:dmh_FE_equation_for_u1} \\
& \mathbf{u}_2 = {\tt M}_2^{-1} \left[{\tt P}_2 {\tt A}_2^{-1} {\tt D}_2^T \widehat{\mathbf{u}}_2
+ \mathbf{g}_2\right]. & \label{eq:dmh_FE_equation_for_u2}
\end{align}
\end{subequations}
We can plug expressions~\eqref{eq:equations_for_U} back into~\eqref{eq:equations_for_J}
to obtain the following affine equations for the degrees of freedom of the dual variable
associated with each element $K \in \mathcal{T}_{h,i}$, $i=1,2$:
\begin{subequations}\label{eq:affine_equations_for_J}
\begin{align}
& \mathbf{J}_1 = {\tt L}_1 \widehat{\mathbf{u}}_1 + \mathbf{b}_1, & \label{eq:affine_equation_for_J1} \\
& \mathbf{J}_2 = {\tt L}_2 \widehat{\mathbf{u}}_2 + \mathbf{b}_2, & \label{eq:affine_equation_for_J2}
\end{align}
where:
\begin{align}
& {\tt L}_1:= - {\tt A}_1^{-1} \left[{\tt N}_1 {\tt M}_1^{-1} {\tt P}_1 {\tt A}_1^{-1}{\tt D}_1^T 
+ {\tt D}_1^T \right], & \label{eq:matrix_L1} \\
& \mathbf{b}_1:= - {\tt A}_1^{-1} {\tt N}_1 {\tt M}_1^{-1} \mathbf{g}_1, & \label{eq:matrix_b1} \\
& {\tt L}_2:= - {\tt A}_2^{-1} \left[{\tt N}_2 {\tt M}_2^{-1} {\tt P}_2 {\tt A}_2^{-1}{\tt D}_2^T 
+ {\tt D}_2^T \right], & \label{eq:matrix_L2} \\
& \mathbf{b}_2:= - {\tt A}_2^{-1} {\tt N}_2 {\tt M}_2^{-1} \mathbf{g}_2. & \label{eq:matrix_b2}
\end{align}
\end{subequations}

\subsubsection{Elimination of the interface Lagrange multipliers $\mathcal{J}_{i,h}$, $i=1,2$}
Restricting the fifth equation in the block linear system~\eqref{eq:dmh_FE_linear_system} 
to the faces belonging to $\mathcal{F}_{h,\Gamma,1}$ yields
\begin{subequations}\label{eq:dmh_FE_jcal}
\begin{align}
& {\tt D}_1 \mathbf{J}_1 - {\tt E}_1^T \mathbf{j}_1 = \mathbf{0}, & \label{eq:dmh_FE_jcal_1} 
\end{align}
whereas the restriction of the sixth equation in the block linear system~\eqref{eq:dmh_FE_linear_system} 
to the faces that belong to $\mathcal{F}_{h,\Gamma,2}$ yields
\begin{align}
& {\tt D}_2 \mathbf{J}_2 - {\tt E}_2^T \mathbf{j}_2 = \mathbf{0}. & \label{eq:dmh_FE_jcal_2} 
\end{align}
Since test functions $\mu_h$ and approximate multipliers $\mathcal{J}_h$ belong to the same 
discrete space $M(F)$ defined in~\eqref{eq:M_F}, equations~\eqref{eq:dmh_FE_jcal_1} and~\eqref{eq:dmh_FE_jcal_2}
are uniquely solvable for each face $F$ belonging to the interface $\Gamma$, and give:
\begin{align}
& \mathbf{j}_1 = ({\tt E}_1^T)^{-1} {\tt D}_1 \mathbf{J}_1,  
& \label{eq:jcal_1} \\
& \mathbf{j}_2 = ({\tt E}_2^T)^{-1} {\tt D}_2 \mathbf{J}_2, 
& \label{eq:jcal_2}
\end{align}
where $\mathbf{J}_1$ and $\mathbf{J}_2$ are given by~\eqref{eq:affine_equations_for_J}.
Also, since functions in the RT0 space~\eqref{eq:Vvect_K} satisfy the property
$$
\mathbf{V}(K) \cdot \mathbf{n}_{\partial K}|_{F \in \partial K} = M(F) \qquad \forall K \in \mathcal{T}_h
\qquad F \in \partial K,
$$
equations~\eqref{eq:jcal_1} and~\eqref{eq:jcal_2} assume the particularly simple form:
\begin{align}
& \mathcal{J}_{1,h}|_F = \mathbf{J}_{1,h} \cdot \mathbf{n}_{1}|_F 
\qquad \forall F \in \mathcal{F}_{h,\Gamma,1}, & \label{eq:jcal_1_bis} \\
& \mathcal{J}_{2,h}|_F = \mathbf{J}_{2,h} \cdot \mathbf{n}_{2}|_F 
\qquad \forall F \in \mathcal{F}_{h,\Gamma,2}. & \label{eq:jcal_2_bis}
\end{align}
\end{subequations}

\subsubsection{Elimination of the hybrid variables on the interface $\Gamma$}
The seventh equation in the block linear system~\eqref{eq:dmh_FE_linear_system} yields
\begin{subequations}\label{eq:dmh_FE_hybrid_lambda}
\begin{align}
& {\tt E}_1 \widehat{\mathbf{u}}_1 - {\tt U}_1^T \boldsymbol{\lambda} = \mathbf{0}, & \label{eq:dmh_FE_uhat_1} 
\end{align}
whereas the eigth equation in the block linear system~\eqref{eq:dmh_FE_linear_system} yields
\begin{align}
& {\tt E}_2 \widehat{\mathbf{u}}_2 - \kappa {\tt U}_2^T \boldsymbol{\lambda} = \mathbf{0}. & \label{eq:dmh_FE_uhat_2} 
\end{align}
\end{subequations}
Using the same argument as for the variable $\mathcal{J}_h$, we see that equations~\eqref{eq:dmh_FE_hybrid_lambda}
are uniquely solvable for each face $F$ belonging to the interface $\Gamma$, and give:
\begin{subequations}\label{eq:dmh_FE_hybrid_lambda_bis}
\begin{align}
& \widehat{\mathbf{u}}_1 = {\tt E}_1^{-1} {\tt U}_1^T \boldsymbol{\lambda},  
& \label{eq:uhat_1} \\
& \widehat{\mathbf{u}}_2 = \kappa {\tt E}_2^{-1} {\tt U}_2^T \boldsymbol{\lambda}. & \label{eq:uhat_2}
\end{align}
We notice that equations~\eqref{eq:dmh_FE_hybrid_lambda_bis} allow to express the segregation 
condition~\eqref{eq:tx_segr} in the DMH formulation in the same manner as in the 3F method.
\end{subequations}

\subsubsection{Construction of the linear algebraic system}
Having expressed the internal variable $\mathbf{J}_h$ in favor of the hybrid variable $\widehat{u}_h$,
the Lagrange multiplier $\mathcal{J}_h$ in favor of $\mathbf{J}_h$ on $\Gamma$ and
the hybrid variable $\widehat{u}_h$ in favor of the Lagrange multiplier $\lambda_h$ on $\Gamma$,
we proceed as follows:
\begin{description}
\item[(step a)] we use the fifth equation in the block linear system~\eqref{eq:dmh_FE_linear_system}
to enforce the interelement continuity of $\mathbf{J}_{1,h} \cdot \mathbf{n}_{1}|_F$
at each $F \in \mathcal{F}_{h,int,1}$ and the boundary condition~\eqref{eq:bcs} 
at each $F \in \mathcal{F}_{h,\Sigma,1}$. 
\item[(step b)] we use the sixth equation in the block linear system~\eqref{eq:dmh_FE_linear_system}
to enforce the interelement continuity of $\mathbf{J}_{2,h} \cdot \mathbf{n}_{2}|_F$
at each $F \in \mathcal{F}_{h,int,2}$ and the boundary condition~\eqref{eq:bcs} 
at each $F \in \mathcal{F}_{h,\Sigma,2}$;
\item[(step c)] we use the ninth equation in the block linear system~\eqref{eq:dmh_FE_linear_system}
to enforce the transmission condition~\eqref{eq:tx_flux} at each $F \in \mathcal{F}_{h,\Gamma}$.
\end{description}

A graphical representation of each of the above three steps is shown in Figure~\ref{fig:flux_continuity}.
\begin{figure}[h!]
\centering
\includegraphics[width=0.8\textwidth]{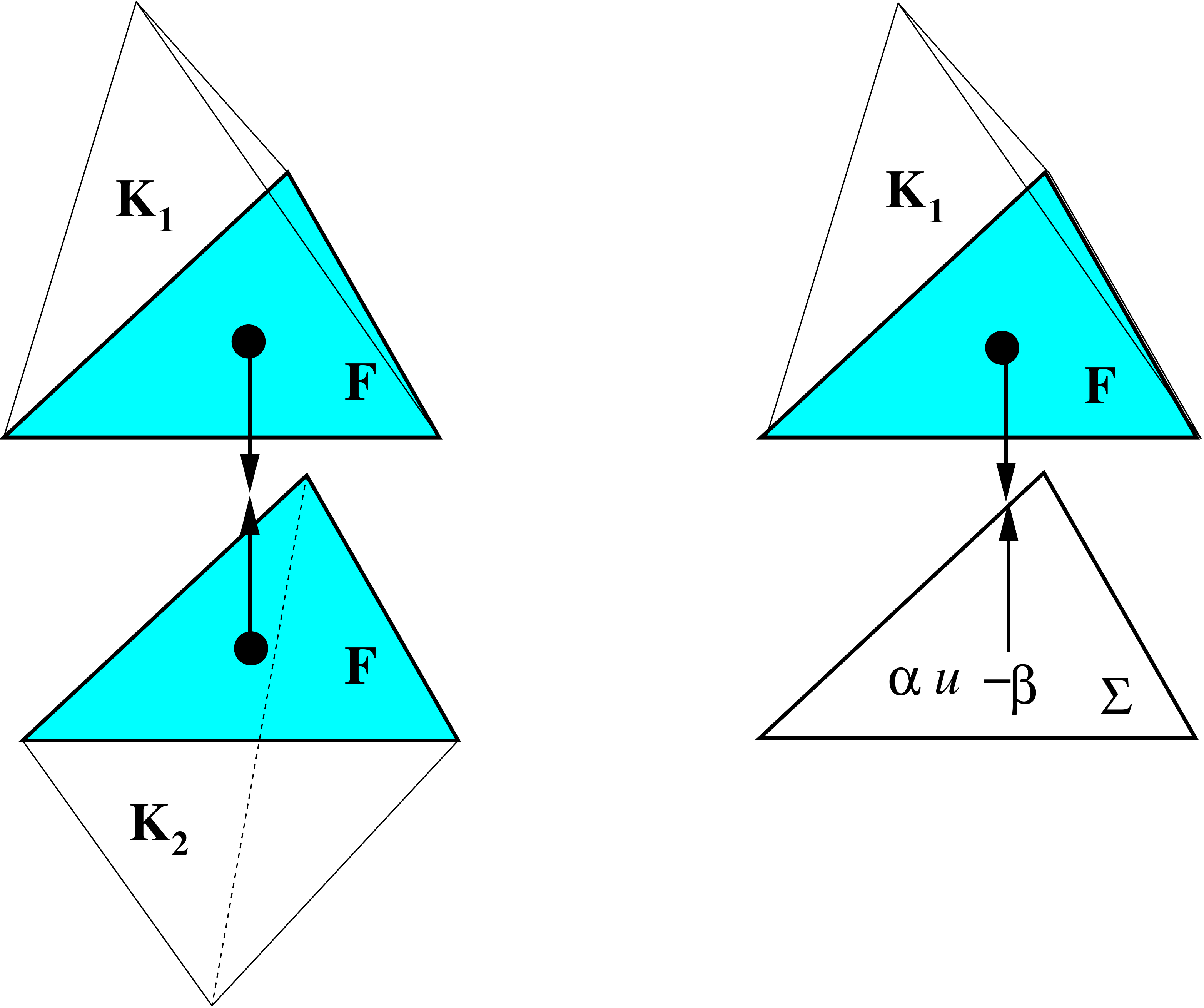}
\caption{Interelement continuity of $\mathbf{J}_{h} \cdot \mathbf{n}|_F$. Left panel: 
$F = \partial K_1 \cap \partial K_2$ (internal face). Right panel: $F = \partial K_1 \cap \Sigma$ 
(boundary face). The arrows represent the degree of freedom of $\mathbf{J}_h|_{K_i}$ associated 
with face $F$ of element $K_i$, $i=1,2$ (left panel) and $i=1$ (right panel). In the case where $F$
is a boundary face the role of $\mathbf{J}_h|_{K_2} \cdot \mathbf{n}$ 
is played by the Robin boundary condition $\alpha u - \beta$.}
\label{fig:flux_continuity}
\end{figure}

The application of the sequence of steps (a), (b) and (c) leads to the construction of 
the following linear reduced system for the DMH-RT0 FEM
\begin{equation}\label{eq:linear_system_reduced_dmh}
{\tt K} \mathbf{U} = \mathbf{t},
\end{equation}
where
$\mathbf{U}\in\mathbb R^{\tt NF}$ is the vector 
of degrees of freedom represented by the values of $\widehat{u}_h$ on each face of
$\mathcal{F}_h$, excluding those belonging to $\Gamma$, and the values of $\lambda_h$ on
each face belonging to $\Gamma$,  ${\tt K} \in \mathbb{R}^{{\tt NF} \times {\tt NF}}$ 
is the stiffness matrix and $\mathbf{t} \in \mathbb{R}^{{\tt NF}}$ is the load vector, with  {\tt NF} denoting the number of faces of $\mathcal{F}_h$.
Each equation in~\eqref{eq:linear_system_reduced_dmh} can be written in explicit form as
\begin{equation}\label{eq:system_explicit}
{\tt K}_{F,F} \mathbf{U}_F + \sum_{G \in {\tt Adj}(F)} {\tt K}_{F,G} \mathbf{U}_G = \mathbf{t}_F \qquad F=1, \ldots, 
{\tt NF},
\end{equation}
where ${\tt Adj}(F)$ denotes the set of faces $G \in \mathcal{F}_h$ that 
have a vertex in common with the closure of $F$. We notice that each row of 
system~\eqref{eq:linear_system_reduced_dmh} corresponding to an internal face $F$ has 7 nonzero entries 
(cf. Figure~\ref{fig:flux_continuity}, left panel)
whereas each row of 
system~\eqref{eq:linear_system_reduced_dmh} 
corresponding to a boundary face $F$ has 4 nonzero entries (cf. Figure~\ref{fig:flux_continuity}, right panel).

\begin{remark}
The unique solvability of~\eqref{eq:linear_system_reduced_dmh} is a consequence of
Lemma~\ref{theorem:wellposedness_dmh_weak_h}. 
\end{remark}

\begin{remark}
The assembly of the stiffness matrix ${\tt K}$ and of the load vector $\mathbf{t}$ 
in~\eqref{eq:linear_system_reduced_dmh} can be conducted  using piecewise linear finite elements for the approximation of the primal variable $u$ as in a standard displacement-based computer code.
In particular, a {\tt for} loop is performed over the elements $K \in \mathcal{T}_h$ and 
for each element the \emph{local} $4 \times 4$ stiffness matrix ${\tt L}_i^K$
and the \emph{local} $4 \times 1$ load vector $\mathbf{t}_i^K=-\mathbf{b}_i^K$, $i=1,2$,
are computed using~\eqref{eq:affine_equations_for_J}. Then, the assembly phase consists of
the following {\tt Matlab} coding:
\begin{verbatim}
for Iloc=1:4,
    I = Lel(K,Iloc);
    for Jloc=1:4
        J = Lel(K,Jloc);
        if (Iloc==Jloc)
           GlobStiffMat(I,I) = GlobStiffMat(I,I) + ...
           LocStiffMat(Iloc,Iloc); 
        else
           GlobStiffMat(I,J) = LocStiffMat(Iloc,Jloc); 
        end
    end
    GlobLoadVec(I) = GlobLoadVec(I) + LocLoadVec(Iloc);
end
\end{verbatim}
In the above code, {\rm K} indicates the global index of element $K$ in the mesh structure,
{\rm Lel} is the connectivity matrix such that {\rm Lel(K,i), i=1,2,3,4} contains the global index of
the face of {\rm K} locally numbered by {\rm i}. In addition,  {\rm GlobStiffMat} and {\rm GlobLoadVec} are 
the global stiffness matrix and global load vector, respectively, whereas {\rm LocStiffMat} and {\rm LocLoadVec} are 
their local counterparts. We notice that the assembly in the DMH-RT0 FEM scheme is performed on a face-oriented basis,
whereas in the standard FEM scheme the assembly is performed on a vertex-oriented basis.
\end{remark}

\subsubsection{Post-processing}

Once the reduced system~\eqref{eq:linear_system_reduced_dmh} is solved, 
the values of $\widehat{u}_h$ on each face of $\mathcal{F}_{h,\Gamma,i}$, $i=1,2$, can be computed
by means of~\eqref{eq:dmh_FE_hybrid_lambda_bis}. Then, the internal 
variables $\mathbf{J}_h$ and  $u_h$ are recovered using~\eqref{eq:affine_equations_for_J} and~\eqref{eq:equations_for_U} over each $K \in \mathcal{T}_h$.

\section{Artificial diffusion stabilization}\label{sec:stabilization}

In this section we describe one of the novel contributions of this article to the theory and development
of dual mixed hybrid methods, namely, the introduction of a stabilization
mechanism that automatically ensures numerical robustness to the scheme 
in the case of advection-dominated regimes, 
a situation that is particularly relevant in the application of problem~\eqref{eq:tx_model} to realistic problems of mass transport in heterogeneous domains.
To quantitatively characterize the weight of advection with respect to the diffusion, for each element $K \in \mathcal{T}_h$, we set 
$\overline{\mathbf{v}}_K:=\mathbf{v}(\mathbf{x}_{B,K})$, where $\mathbf{x}_{B,K}$ 
is the barycenter of $K$, and we define the \emph{local P\`eclet number} as
\begin{align}
& {\rm Pe}_K:= \max_{i=1, \ldots, 6} \frac{|\overline{\mathbf{v}}_K \cdot \mathbf{e}_i|}{2 \mu}, 
\label{eq:local_peclet}
\end{align}
where $\mathbf{e}_i$, $i=1,\ldots, 6$, is the vector connecting two vertices of $K$.
Relation~\eqref{eq:local_peclet} extends to the case of tetrahedral elements the definition 
given in~\cite{fenics_web_page} in the case of triangular elements.
If ${\rm Pe}_K < 1$ the problem is locally
diffusion-dominated whereas if ${\rm Pe}_K > 1$ the problem is locally advection-dominated.
In this latter case, an effective 
approach to prevent the onset of numerical instabilities
consists of introducing an \emph{artificial diffusion tensor} $\boldsymbol{\mu}^{\ast}_K$
constructed in such a way to locally increase the diffusion mechanism. 
Following~\cite{Bank1990,onate2000_cimne} and~\cite[Chapter~6]{quarteroni1994numerical}, 
the modified diffusion tensor to be used in the artificial diffusion method is defined as 
\begin{align}
& (\boldsymbol{\mu}_h)_K = \boldsymbol{\mu}_K + \boldsymbol{\mu}^{\ast}_K 
= \mu \mathbf{I} + \boldsymbol{\mu}^{\ast}_K. \label{eq:modified_diffusion_tensor}
\end{align}
The effect of numerical dissipation is minimized if artificial diffusion is added 
\emph{only} in the streamline direction,
as done in the Streamline Upwind Petrov-Galerkin method introduced in~\cite{BROOKS1982}.
To follow this approach, if $|\overline{\mathbf{v}}_K| \neq 0$, we define the \emph{streamline unit vector}
\begin{subequations}\label{eq:stabilized_diffusion}
\begin{align}
& \boldsymbol{\beta}_K:= \frac{\overline{\mathbf{v}}_K}{|\overline{\mathbf{v}}_K|}
\label{eq:streamline_unit_vector}
\end{align}
and set
\begin{align}
& \boldsymbol{\mu}^{\ast}_K:= \mu \Phi({\rm Pe}_K) 
\boldsymbol{\beta}_K \boldsymbol{\beta}_K^T. \label{eq:artificial_diffusion} 
\end{align}
The amount of artificial diffusion depends on the \emph{stabilization function} $\Phi$
that is required to satisfy the following properties:
\begin{align}
& \Phi(X) > 0 \qquad \forall X > 0, \label{eq:prop_1_Phi} \\
& \lim_{X \rightarrow 0^+} \Phi(X) = 0^+. \label{eq:prop_2_Phi}
\end{align}
We refer to~\cite{roos2008robust} for a detailed illustration and analysis of several 
choices for $\Phi$. In the numerical examples reported in Section~\ref{sec:numerical_results} we use
the following form of the stabilization function
\begin{align}
& \Phi(X):= X - 1 + {\tt Be}(2X), \label{eq:Phi_stab_SG}
\end{align}
where ${\tt Be}(t):= t/(e^t-1)$ is the inverse of the Bernoulli function. The choice~\eqref{eq:Phi_stab_SG}
satisfies properties~\eqref{eq:prop_1_Phi}-~\eqref{eq:prop_2_Phi}, and in particular it can be seen that
\begin{align}
& \lim_{h_K \rightarrow 0} \Phi({\rm Pe}_K) = \mathcal{O}(h_K^2). \label{eq:limit_Phi_h_0}
\end{align}
The above relation shows that the artificial diffusion based on~\eqref{eq:Phi_stab_SG} 
decreases quadratically as the mesh size becomes small, and because of this asymptotic behavior 
the choice~\eqref{eq:Phi_stab_SG} is referred to as \emph{optimal} artificial diffusion
(see~\cite{BROOKS1982} and~\cite{onate2000_cimne}).
Another popular choice of $\Phi$, that is also implemented in the numerical examples reported in Section~\ref{sec:numerical_results}, is the so-called Upwind method for which 
\begin{align}
& \Phi(X):= X. \label{eq:Phi_stab_upwind}
\end{align}
The upwind stabilization based on~\eqref{eq:Phi_stab_upwind} 
introduces an artificial diffusion that decreases only linearly 
as the mesh size becomes small, therefore worsening the accuracy of the computed solution.
On the other hand, when ${\rm Pe}_K$ becomes large, the optimal artificial diffusion and upwind stabilizations
practically coincide, thereby supporting the use of~\eqref{eq:Phi_stab_SG} in all regimes
instead of~\eqref{eq:Phi_stab_upwind} (see also~\cite{Bank1990} for further discussion of this issue).
\end{subequations}

\section{Spectral analysis of the stabilized diffusion tensor}\label{sec:spectral_analysis_stab}

In this section we study the spectrum of the stabilized
diffusion tensor~\eqref{eq:modified_diffusion_tensor} 
as a function of the transport parameters that characterize the problem at hand.
The analysis is carried out for the stabilization function~\eqref{eq:Phi_stab_SG} but similar considerations
apply to the stabilization function~\eqref{eq:Phi_stab_upwind}.
Denoting by $\Lambda_i$ and $\mathbf{X}_i$, $i=1,2,3$, the eigenvalues and the corresponding
eigenvectors of $(\boldsymbol{\mu}_h)_K$, an explicit computation yields
\begin{subequations}\label{eq:spectrum_mu_h_tensor}
\begin{align}
& \Lambda_1 = \Lambda_2 = \mu, \, \, \Lambda_3 = \mu(1+ \Phi({\rm Pe}_K)), \label{eq:eig_stab} \\
& \mathbf{X}_1 = \left[ -\frac{\beta_{K,y}}{\beta_{K,x}}, \, 1, \, 0 \right]^T, \label{eq:eigenvec_1_stab} \\
& \mathbf{X}_2 = \left[ -\frac{\beta_{K,z}}{\beta_{K,x}}, \, 0, \, 1 \right]^T, \label{eq:eigenvec_2_stab} \\
& \mathbf{X}_3 = \left[ \frac{\beta_{K,x}}{\beta_{K,z}}, \, \frac{\beta_{K,y}}{\beta_{K,z}}, \, 1 \right]^T. 
\label{eq:eigenvec_3_stab}
\end{align}
The stabilized diffusion tensor is a symmetric positive definite $3 \times 3$ matrix. 
Replacing~\eqref{eq:Phi_stab_SG} into the expression of $\Lambda_3$ we obtain
\begin{align}
& \Lambda_3 = \mu ({\rm Pe}_K + {\tt Be}(2 {\rm Pe}_K)). \label{eq:lambda_3_stab}
\end{align} 
If the local P\`eclet number is very small, a Taylor expansion of ${\tt Be}(2 {\rm Pe}_K)$ 
in the neighbourhood of 0 yields $\Lambda_3 = \mu$, so that $(\boldsymbol{\mu}_h)_K$ coincides
with $\boldsymbol{\mu}_K = \mu \mathbf{I}$, as expected, because the problem is \emph{not} 
advection-dominated and thus no stabilization is actually needed. Conversely, if the local P\`eclet 
number is much larger than 1 the quantity ${\tt Be}(2 {\rm Pe}_K)$ can be neglected 
in~\eqref{eq:lambda_3_stab}, yielding
\begin{align}
& \Lambda_3 \simeq \mu {\rm Pe}_K \gg \left\{\Lambda_1, \, \Lambda_2\right\}. \label{eq:lambda_3_stab_max}
\end{align} 
Therefore, in the case where problem~\eqref{eq:tx_model} is locally advection-dominated the three-dimensional
surface representing the spectrum of the stabilized diffusion tensor in the euclidean space $\mathbb{R}^3$
is an ellipsoid centered in the origin, with the $x_1$ and $x_2$ principal axes of equal length 
and with a strongly elongated principal axis $x_3$.
\end{subequations}

\begin{example}
Consider the reference tetrahedron with vertices $[0,0,0]^T$, $[1,0,0]^T$, $[0,1,0]^T$ and $[0,0,1]^T$. 
Assume that $\mathbf{v}_K = [0, \, 0, \, 1]^T$ and that $\mu=10^{-2}$. Using~\eqref{eq:local_peclet}
we get ${\rm Pe}_K=50$, which means that the model is in the advection-dominated regime. The 
artificial diffusion tensor is
$$
\boldsymbol{\mu}^{\ast}_K =
\left[
\begin{array}{lll}
0 & 0 & 0 \\
0 & 0 & 0 \\
0 & 0 & \mu \Phi(2 {\rm Pe}_K) \simeq 0.49
\end{array}
\right],
$$
which shows that the stabilization introduces a contribution \emph{only} along the $z$ axis
that is the streamline direction.
\end{example}

\section{Numerical results}\label{sec:numerical_results}
In this section we perform a thorough validation of the performance of the proposed method.
To this end, we have implemented problem~\eqref{eq:tx_model} 
and the DMH-RT0 FEM scheme proposed for its discretization
within the computational software MP-FEMOS 
(Multi-Physics Finite Element Modeling Oriented Simulator) 
that has been developed by one of the authors~\cite{Mauri2014,Mauri2014currents,Airoldi2015,Mauri_JMO_2016,SACCO_amm_2017}. 
MP-FEMOS is a general-purpose modular code based on the Galerkin Finite Element Method
that is programmed in a fully 3D framework through shared libraries using an object-oriented language (C++).  
Several situations are addressed. In Section~\ref{sec:study_accuracy} the accuracy of the scheme is studied in
two different cases, corresponding to non active and active interface.
In Section~\ref{sec:study_stabilization} the stability of the scheme is studied in different regimes,
corresponding to low and high local P\`eclet numbers.
In all test cases, the simulation domain is the unit cube $(0,1)\times(0,1)\times(0,1)$ with the 
interface at $z=0.5$. Dirichlet boundary conditions are applied at the bottom and top faces of the cube,
with $u=0$ at the bottom and $u=1$ at the top, whereas 
homogeneous Neumann conditions are imposed for $\mathbf{J}$ on the lateral surface.
\begin{figure}[h!]
\centering
\includegraphics[width=0.6\textwidth,height=0.6\textwidth]{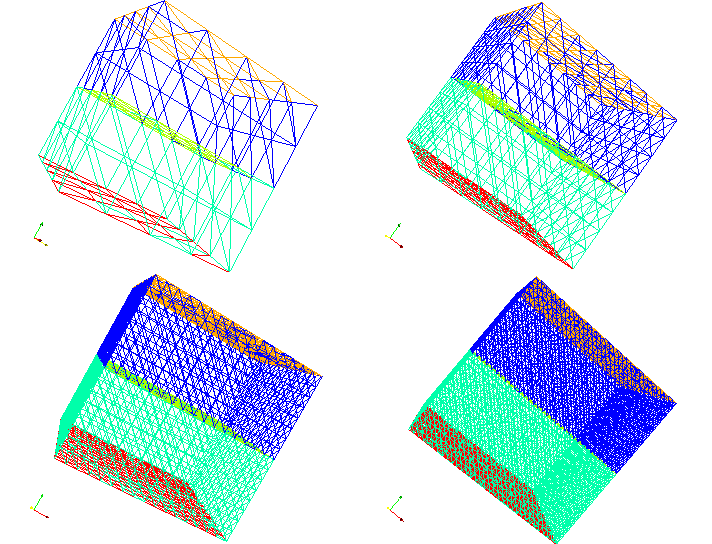}
\caption{Regular triangulations with $h = [0.4330,  \,   0.2165,  \,  0.1083, \,    0.0541]$.}
\label{fig:meshes}
\end{figure}

In the computational examples illustrated in Section~\ref{sec:study_accuracy}, 
the four tetrahedral meshes shown in Figure~\ref{fig:meshes} are used. Partitions 
are made of regular elements, with 
$h = [0.4330,  \,   0.2165,  \,  0.1083, \,    0.0541]$.
For any piecewise smooth function $\eta_h: \mathcal{T}_h \rightarrow \mathbb{R}$, we set
$$
\| \eta_h \|_{\infty,h}:= \max_{K \in \mathcal{T}_h} |\eta_h(\mathbf{x}_{B,K})|.
$$ 
\subsection{Convergence analysis}\label{sec:study_accuracy}
In Section~\ref{sec:noactive_accuracy} we study the accuracy of the DMH-RT0 FEM scheme in the case 
where both $u$ and $\mathbf{J} \cdot \mathbf{n}$ are continuous at $\Gamma$.
In Section~\ref{sec:active_accuracy} we consider the case 
where both $u$ and $\mathbf{J} \cdot \mathbf{n}$ are discontinuous at $\Gamma$.
All test cases considered in Section~\ref{sec:study_accuracy} are conducted in a regime where the P\`eclet number is less than 1 and so the scheme is implemented without stabilization.
The effect of stabilization will be assessed in Section~\ref{sec:study_stabilization}.

\subsubsection{Non active interface}\label{sec:noactive_accuracy}
Let us set $\mu=1$, $r=1$, $g=1$, $\mathbf v = v_z \mathbf e_3$, $v_z=1$, $\kappa=1$ and $\sigma=0$, 
where $\mathbf e_3$ denotes the unit vector of the $z$-axis shown in Figure~\ref{fig:error_u_and_J_noactive}. 
The exact solution of system~\eqref{eq:tx_model} is the pair:
\begin{align}
& u(z) = \Frac{g}{r} + C_1 e^{\lambda_1 z} + C_2 e^{\lambda_2 z} & \qquad z \in [0,1] \label{eq:u_exact_1} \\
& J(z) = \Frac{v_z g}{r} + C_1 e^{\lambda_1 z}\left( v_z - \mu \lambda_1 \right) 
+ C_2 e^{\lambda_2 z}\left( v_z - \mu \lambda_2 \right) &  \qquad z \in [0,1], \label{eq:J_exact_1}
\end{align}
where $\lambda_{1,2} = (v_z \pm \sqrt{v_z^2 + 4 r \mu})/(2 \mu)$ and
\begin{align*}
& C_1 = \Frac{-1 + \Frac{g}{r}(1-e^{\lambda_2})}{e^{\lambda_2} - e^{\lambda_1}}, \qquad 
C_2 = \Frac{1 - \Frac{g}{r}(1-e^{\lambda_1})}{e^{\lambda_2} - e^{\lambda_1}}. 
\end{align*}
Figure~\ref{fig:error_u_and_J_noactive} (left panel) illustrates the errors associated with
the scalar variable $u$ whereas Figure~\ref{fig:error_u_and_J_noactive} (right panel)
shows the errors associated with the vector variable $\mathbf{J}$.
Results indicate that: (i) the DMH formulation is linearly converging with respect to (w.r.t.) 
the graph norm in the $L^2\times H({\rm div})$-topology; (ii) $u_h$ quadratically converges to the value of 
$u$ at the barycenters of $\mathcal{T}_h$; and (iii) $\widehat{u}_h$ 
quadratically converges to the value of $u$ w.r.t. 
the discrete maximum norm and the $L^2$ norm. These outcomes are in complete agreement with 
the theoretical estimates of Section~\ref{sec:DMH_FE_method} and with existing
theoretical estimates for the DMH formulation applied to the solution of elliptic boundary value problems
on a single domain (see~\cite{ArnoldBrezzi1985,douglas_roberts_1985,brezzifortin1991,arbogast_and_chen_1995}).
\begin{figure}[h!]
\centering
\includegraphics[width=0.48\textwidth,height=0.5\textwidth]{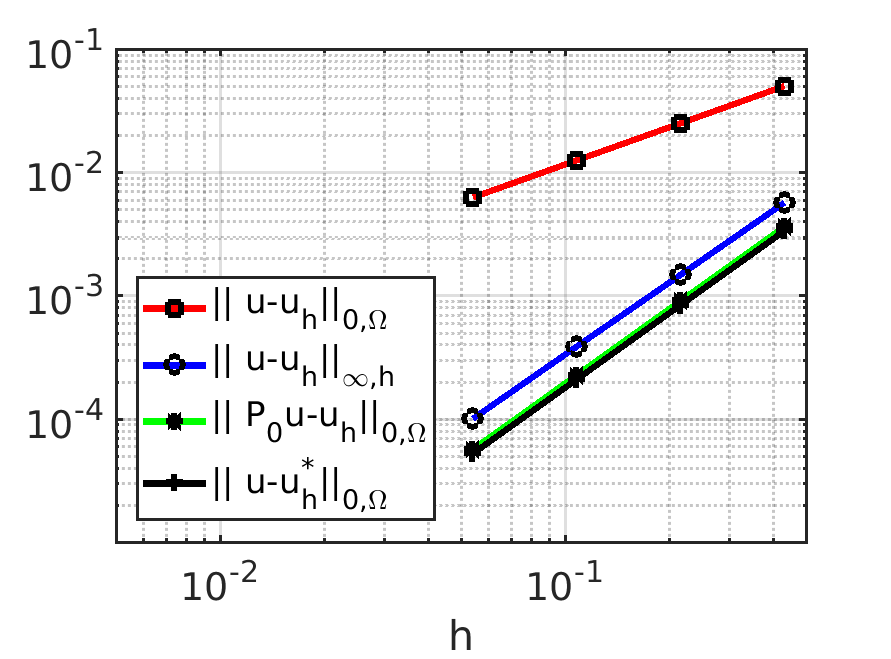}
\includegraphics[width=0.48\textwidth,height=0.5\textwidth]{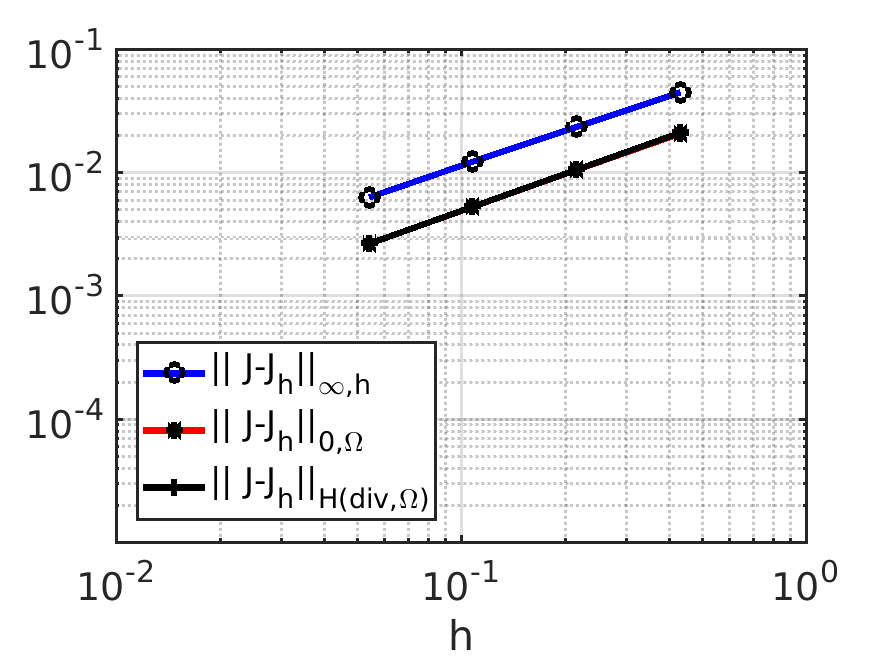}
\caption{Error curves for the DMH method. The values of model coefficients are: 
$\mu=1$, $r=1$, $g=1$, $\mathbf v = v_z \mathbf e_3$, $v_z=1$, $\kappa=1$ and $\sigma=0$. Left panel:
$\| u - u_h \|_{0, \Omega}$ (red curve);
$\| u - u_h\|_{\infty,h}$ (blue curve); $\| P_0 u - u_h \|_{0, \Omega}$ (green curve); 
$\| u - u_h^\ast \|_{0, \Omega}$ (black curve). Right panel: 
$\| \mathbf{J} - \mathbf{J}_h\|_{\infty,h}$ (blue curve); 
$\| \mathbf{J} - \mathbf{J}_h\|_{0,\Omega}$ (red curve); 
$\| \mathbf{J} - \mathbf{J}_h\|_{H({\rm div};\Omega)}$ (black curve).}
\label{fig:error_u_and_J_noactive}
\end{figure}

\subsubsection{Active interface}\label{sec:active_accuracy}

Let us set $\mu=1$, $r=1$, $g=1$, $\mathbf v = v_z \mathbf e_3$, $v_z=1$, as in the previous section, and let us set 
$\kappa=2$ and $\sigma=1$ to model the active interface. 
The exact solution of system~\eqref{eq:tx_model} is the pair:
\begin{align}
& u(z) = \Frac{g}{r} + C_1 e^{\lambda_1 z} + C_2 e^{\lambda_2 z} & \qquad z \in [0,0.5) \label{eq:u_exact_2_left} \\
& J(z) = \Frac{v_z g}{r} + C_1 e^{\lambda_1 z}\left( v_z - \mu \lambda_1 \right) 
+ C_2 e^{\lambda_2 z}\left( v_z - \mu \lambda_2 \right) &  \qquad z \in [0,0.5), \label{eq:J_exact_2_left} \\
& u(z) = \Frac{g}{r} + C_3 e^{\lambda_1 z} + C_4 e^{\lambda_2 z} & \qquad z \in [0.5, 1] \label{eq:u_exact_2_right} \\
& J(z) = \Frac{v_z g}{r} + C_3 e^{\lambda_1 z}\left( v_z - \mu \lambda_1 \right) 
+ C_4 e^{\lambda_2 z}\left( v_z - \mu \lambda_2 \right) &  \qquad z \in [0.5,1], \label{eq:J_exact_2_right}
\end{align}
where $\lambda_{1,2} = (v_z \pm \sqrt{v_z^2 + 4 r \mu})/(2 \mu)$ and
the four constants $C_k$, $k=1,2,3,4$ are the solutions of the following linear system
\begin{align*}
& \mathcal{C} \mathbf{c} = \mathbf{g}, & 
\end{align*}
where $\mathbf{c} = [C_1, C_2, C_3, C_4]^T$, 
\begin{eqnarray*}
\mathcal{C} =
\left[
\begin{array}{llll}
1 & 1 & 0 & 0 \\
e^{\lambda_1/2}(v_z - \mu \lambda_1) & e^{\lambda_2/2}(v_z - \mu \lambda_2) &
- e^{\lambda_1/2}(v_z - \mu \lambda_1) & - e^{\lambda_2/2}(v_z - \mu \lambda_2) \\
- \kappa e^{\lambda_1/2} & - \kappa e^{\lambda_2/2} & e^{\lambda_1/2} & e^{\lambda_2/2} \\
0 & 0 & e^{\lambda_1} & e^{\lambda_2}
\end{array}
\right],
\end{eqnarray*}
and
\begin{eqnarray*}
\mathbf{g} =
\left[
\begin{array}{l}
- \Frac{g}{r} \\[2mm]
-\sigma \\
(\kappa -1) \Frac{g}{r} \\
1 - \Frac{g}{r}
\end{array}
\right].
\end{eqnarray*}
The error curves obtained for this
problem are shown in Figure~\ref{fig:error_u_and_J_active}. Results are very similar to 
those obtained in the case of a nonactive interface.
\begin{figure}[h!]
\centering
\includegraphics[width=0.48\textwidth,height=0.5\textwidth]{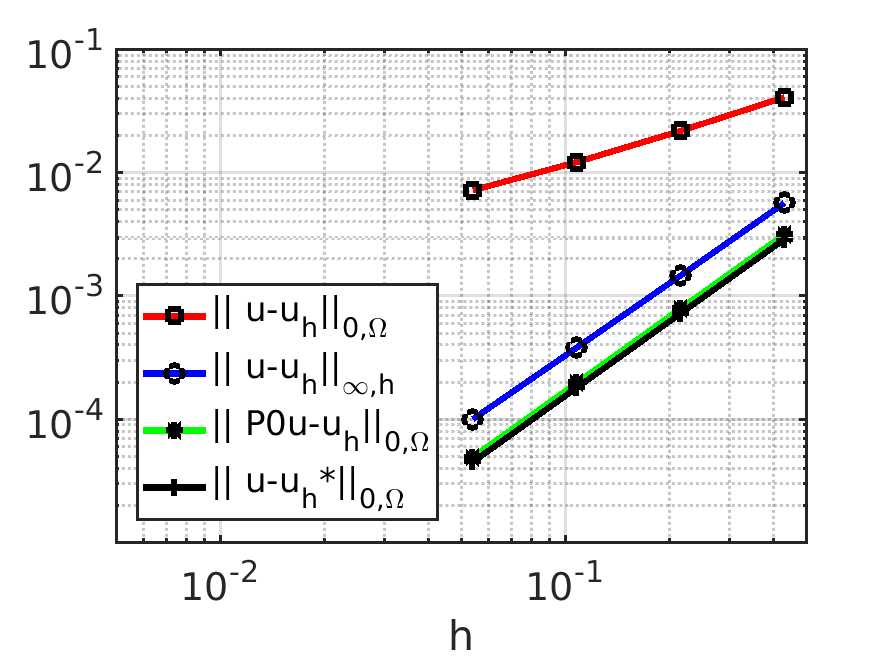}
\includegraphics[width=0.48\textwidth,height=0.5\textwidth]{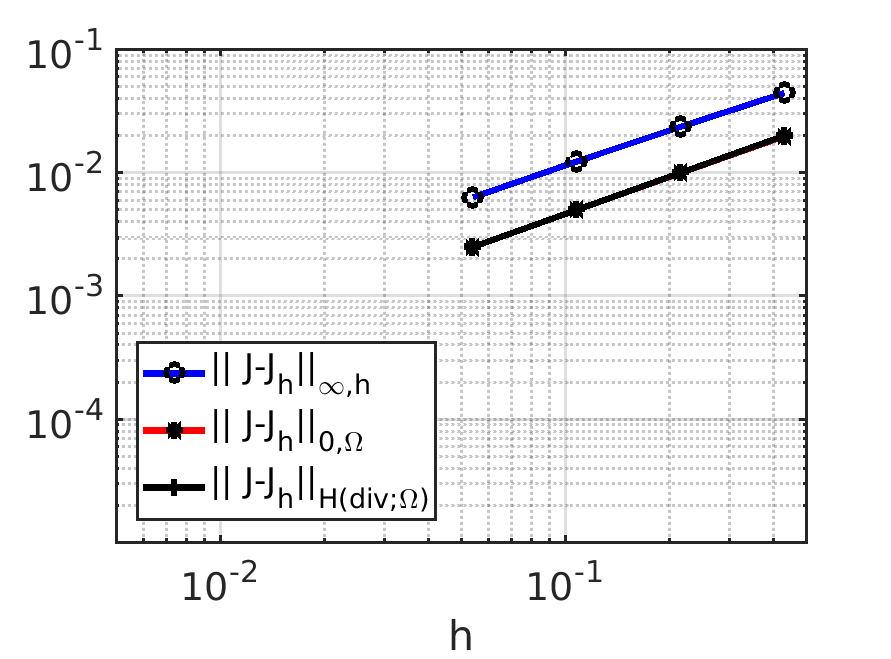}
\caption{Error curves for the DMH-RT0 FEM. The values of model coefficients are: 
$\mu=1$, $r=1$, $g=1$, $\mathbf v = v_z \mathbf e_3$, $v_z=1$, $\kappa=2$ and $\sigma=1$. Left panel:
$\| u - u_h \|_{0, \Omega}$ (red curve);
$\| u - u_h\|_{\infty,h}$ (blue curve); $\| P_0 u - u_h \|_{0, \Omega}$ (green curve); 
$\| u - u_h^\ast \|_{0, \Omega}$ (black curve). Right panel: 
$\| \mathbf{J} - \mathbf{J}_h\|_{\infty,h}$ (blue curve); 
$\| \mathbf{J} - \mathbf{J}_h\|_{0,\Omega}$ (red curve); 
$\| \mathbf{J} - \mathbf{J}_h\|_{H({\rm div};\Omega)}$ (black curve).}
\label{fig:error_u_and_J_active}
\end{figure}

A three-dimensional view of the solutions $u_h$ and $J_{h,z}$ computed by the 
DMH-RT0 FEM scheme is reported in 
Figure~\ref{fig:solutions_active_3D} whereas Figure~\ref{fig:solutions_active_1D} shows a 
a cross-sectional view of the along the $z$-axis of the same computed quantities. 
Results indicate that the method is able to accurately capture the jump discontinuity even with a 
rather coarse partition of the domain.
\begin{figure}[h!]
\centering
\includegraphics[width=0.8\textwidth,height=0.6\textwidth]{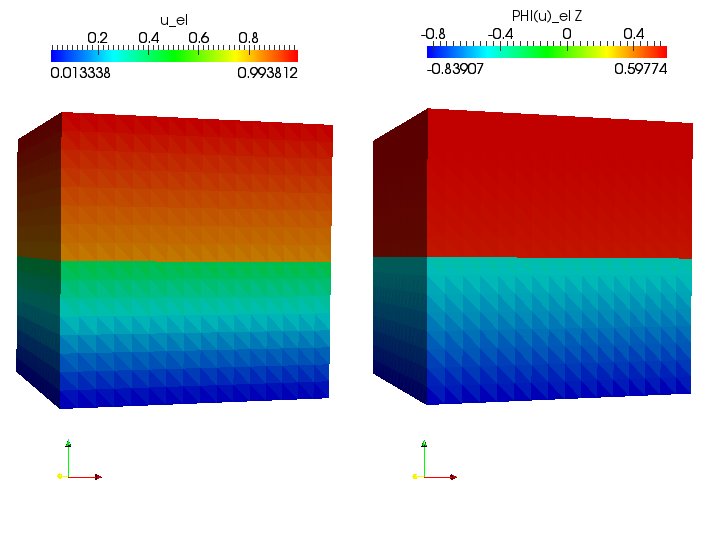}
\caption{3D color plots of the solutions computed by the DMH-RT0 FEM scheme. 
Left panel: values of $u_h$ at the barycenter of each element. 
Right panel: values of $J_{z,h}$ at the barycenter of each element.}
\label{fig:solutions_active_3D}
\end{figure}

\begin{figure}[h!]
\centering
\includegraphics[width=0.48\textwidth,height=0.5\textwidth]{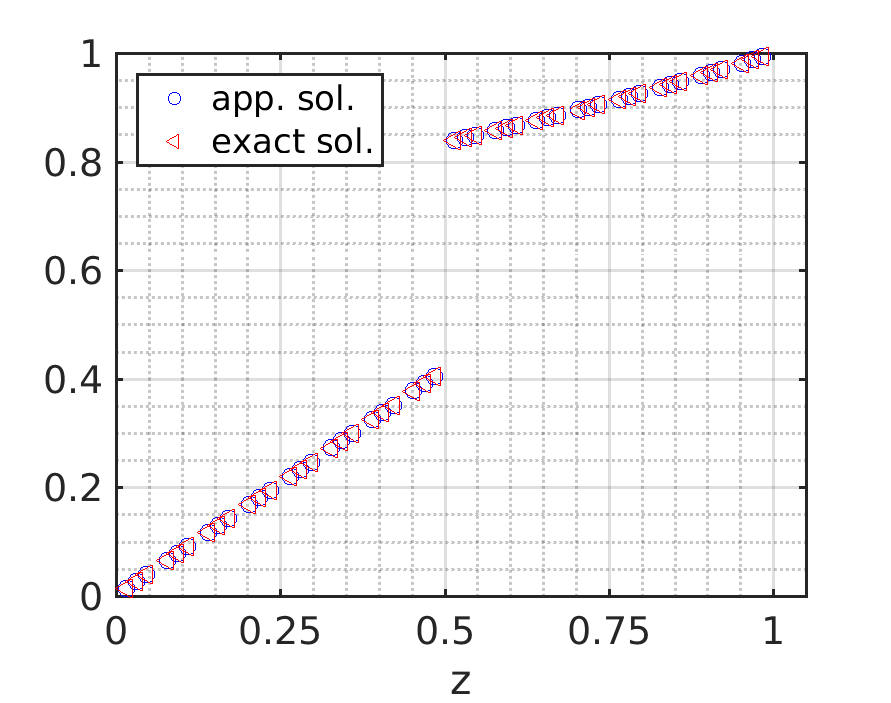}
\includegraphics[width=0.48\textwidth,height=0.5\textwidth]{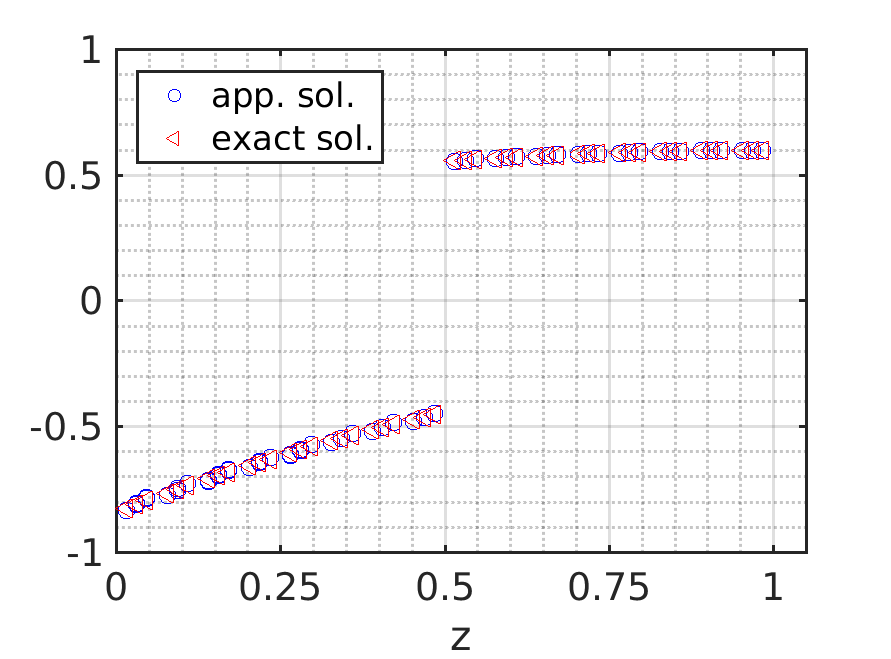}
\caption{1D plots of the solutions computed by the DMH-RT0 FEM scheme along the $z$-axis.
Left panel: values of $u_h$ at the barycenter of each element. 
Right panel: values of $J_{z,h}$ at the barycenter of each element.}
\label{fig:solutions_active_1D}
\end{figure}

\subsection{The effect of stabilization}\label{sec:study_stabilization}
In this section we carry out a verification of the effect of the streamline artificial
diffusion on the stability properties of the DMH method in the presence of a dominating advective term.
The tetrahedral mesh is the same in all the tested cases with $h = 0.108253$.

\subsubsection{Non active interface}\label{sec:noactive}
Here we consider the case where the interface is not active and so we consider the same parameter choice as in Section~\ref{sec:noactive_accuracy} where, in particular, $\kappa=1$ and $\sigma=0$.
Figure~\ref{fig:u_h_star} shows a cross-sectional view of the reconstructed solution $u^\ast_h$ along the $z$-axis
in correspondance of six increasing values of the local P\`eclet number obtained with the following data:
value nr. 1: $\mu=0.5$, $v_z=1$; value nr. 2: $\mu=0.0125$, $v_z = 0.625$; values from nr. 3 to nr. 6:
$v_z = 0.625$ and $\mu = \left\{ 6.25 \cdot 10^{-3}, 3.125 \cdot 10^{-3}, 1.5625 \cdot 10^{-3}, 
7.8125 \cdot 10^{-4} \right\}$. 
\begin{figure}[h!]
\includegraphics[width=0.32\textwidth,height=0.3\textwidth]{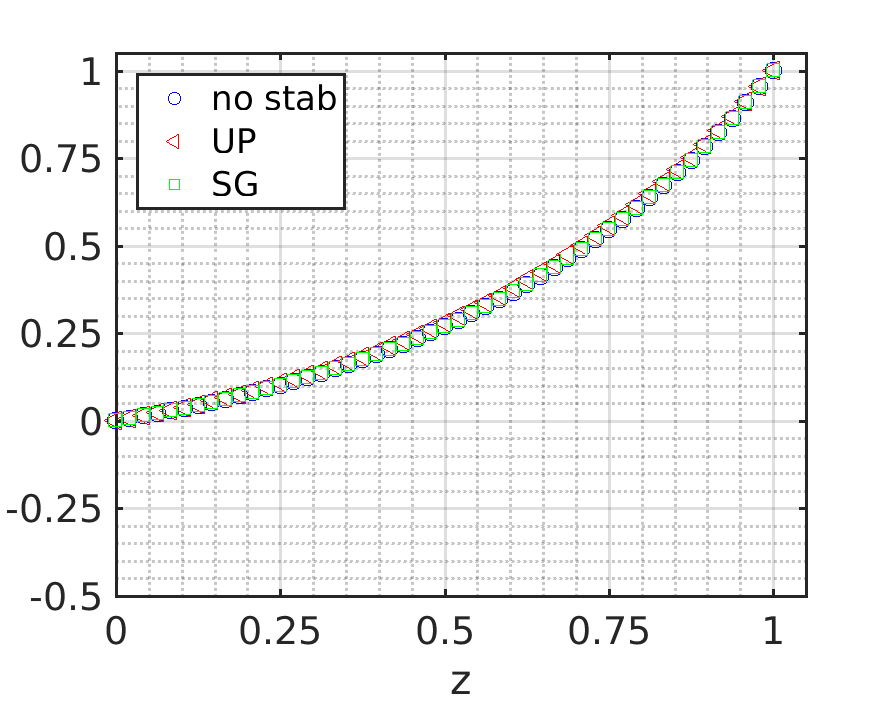}
\includegraphics[width=0.32\textwidth,height=0.3\textwidth]{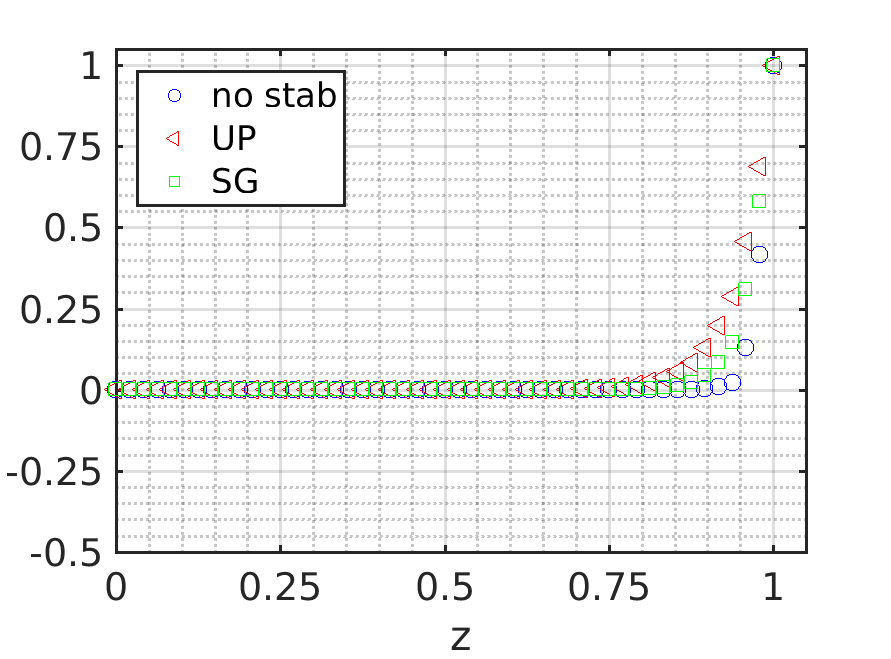}
\includegraphics[width=0.32\textwidth,height=0.3\textwidth]{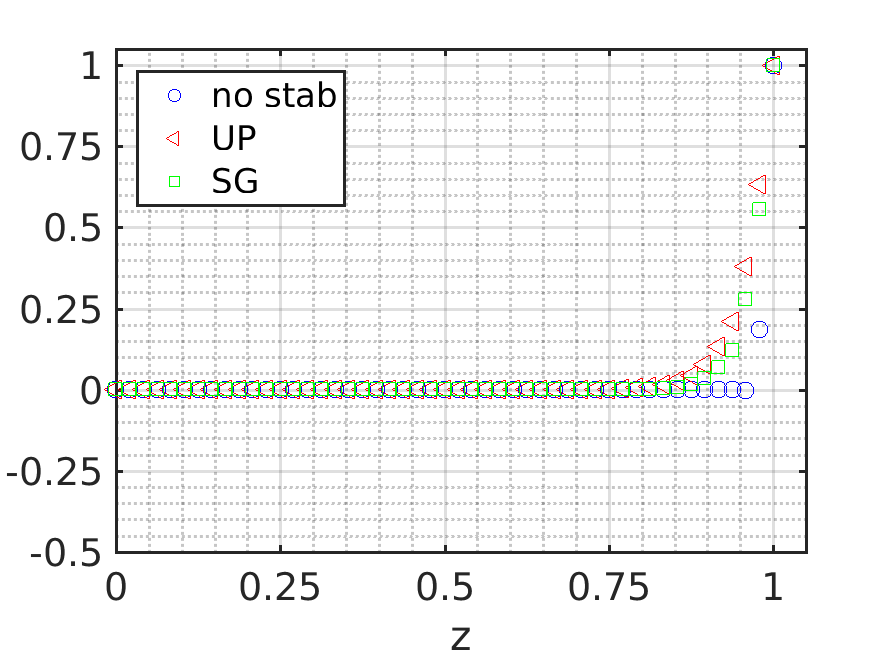}
\\
\vspace{-6pt}
\includegraphics[width=0.32\textwidth,height=0.3\textwidth]{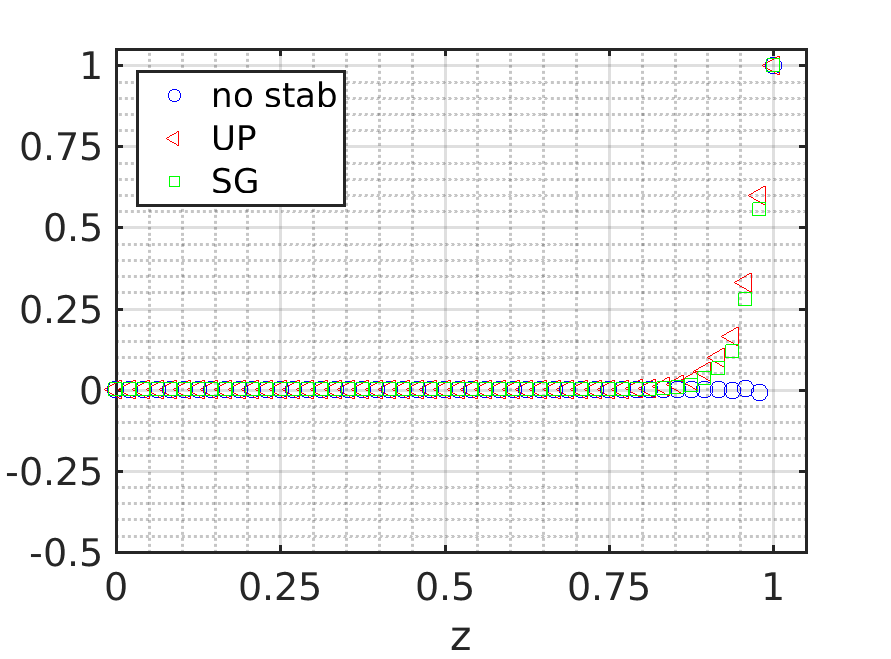}
\includegraphics[width=0.32\textwidth,height=0.3\textwidth]{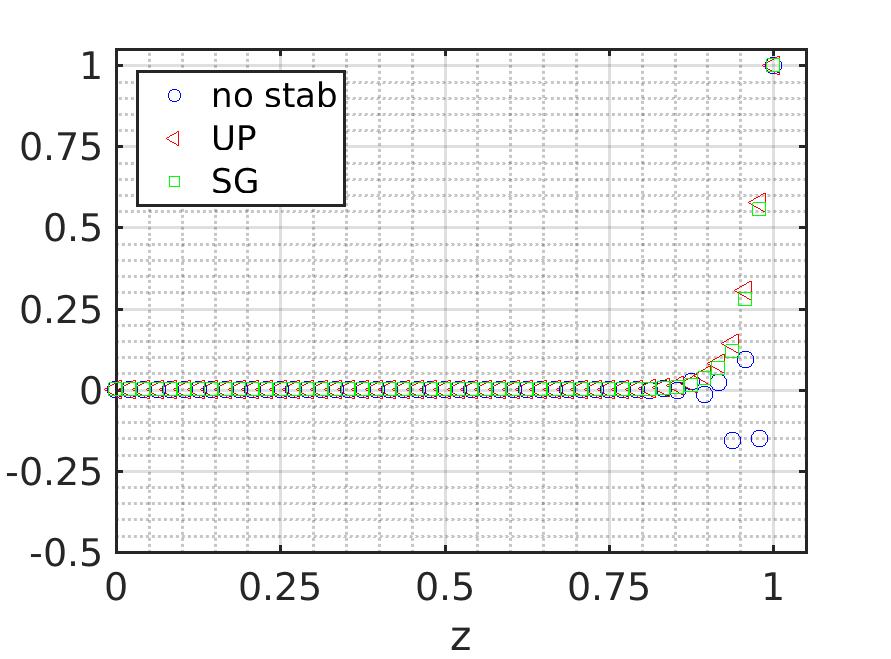}
\includegraphics[width=0.32\textwidth,height=0.3\textwidth]{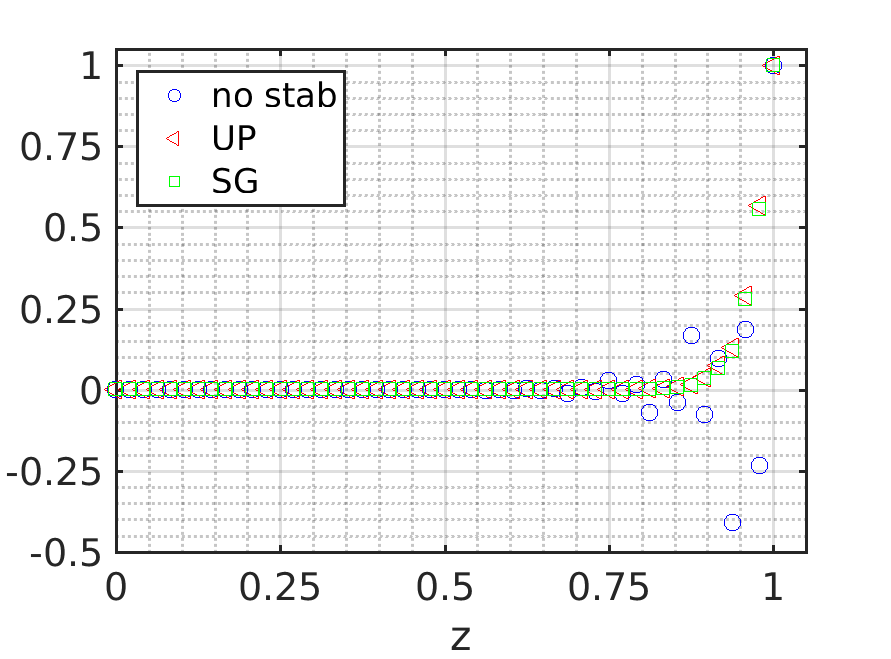}
\caption{1D plot of the computed solution $u^\ast_h$ along the $z$-axis. 
Blue circles: unstabilized solution. Green squares: SG stabilization. Red triangles: Upwind stabilization. Top row.
left panel: ${\rm Pe}_K=0.1083$; middle panel: ${\rm Pe}_K=2.7063$; right panel: ${\rm Pe}_K=5.4127$.
Bottom row. left panel: ${\rm Pe}_K=10.8253$; middle panel: ${\rm Pe}_K=21.6506$; 
right panel: ${\rm Pe}_K=42.3012$.}
\label{fig:u_h_star}
\end{figure}

Results show that as ${\rm Pe}_K$ increases, the non stabilized method starts to display 
spurious unphysical oscillations in the boundary layer region, which tend to propagate backwards 
throughout the computational domain because of the markedly hyperbolic behavior of the problem.
On the contrary, the stabilized method is characterized by a robust behavior
with respect to the increase of the local P\`eclet number, showing in particular that 
the SG stabilized DMH method computes a solution that is much more accurate than that 
computed by the Upwind stabilized in the boundary layer region.

\subsubsection{Active interface}\label{sec:active}
In this section, we assume that the interface located
at $z=0.5$ is active and set $\kappa=2$ and $\sigma=1$.
Moreover, the values of model coefficients are selected in such a way that the problem is diffusion-dominated in one subregion and advection-dominated in the other region.
Specifically, in the first case of study we set
$\mu= 1$, $\mathbf v = v_z \mathbf e_3$, $v_z=1$, in $\Omega_1$ and $\mu=0.0325$, 
$\mathbf v = v_z \mathbf e_3$ in $\Omega_2$, so that
${\rm Pe}_K|_{\Omega_1} = 0.0541$ and ${\rm Pe}_K|_{\Omega_2} = 1.6654$. In the second case of 
study we set $\mu=1$, $v_z = 1$ in $\Omega_1$ and $\mu=0.008125$, $v_z = 1$ in $\Omega_2$, so that
${\rm Pe}_K|_{\Omega_1} = 0.0541$ and ${\rm Pe}_K|_{\Omega_2} = 6.6618$.
Thus, in both cases of study the problem is diffusion-dominated in $\Omega_1$ and advection-dominated 
in $\Omega_2$. 
Figure~\ref{fig:inst_disc} shows a cross-sectional view of the reconstructed solution $u^\ast_h$ 
along the $z$-axis. In both cases we see that: (a) the non stabilized and stabilized solutions
correctly capture the sharp discontinuity at $z=0.5$; (b) the non stabilized solution displays
increasing instabilities in the boundary layer region at $z=1$ as ${\rm Pe}_K$ increases;
(c) the two stabilized solutions capture the boundary layer without any unphysical oscillations.

\begin{figure}[h!]
\centering
\includegraphics[width=0.48\textwidth,height=0.48\textwidth]{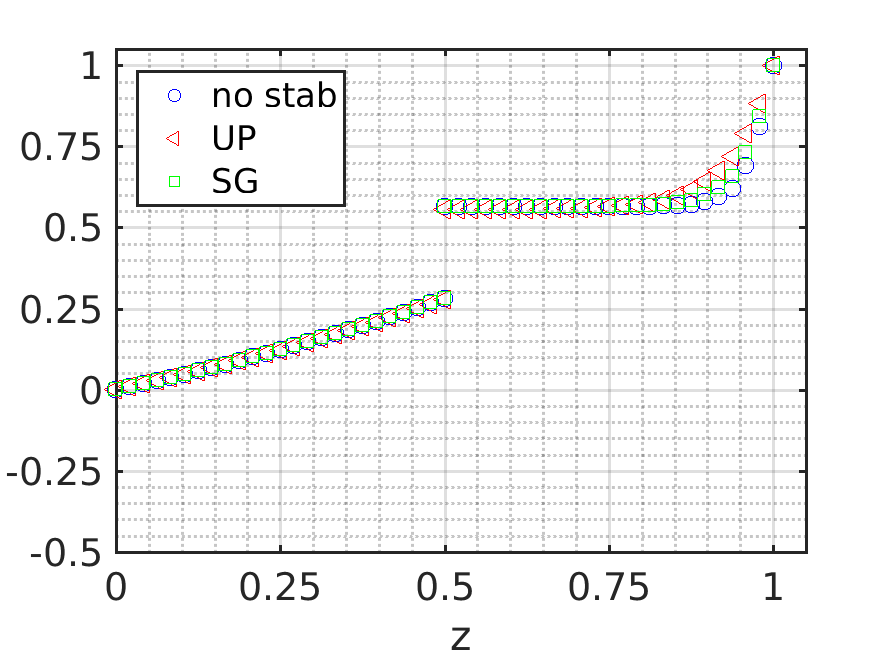}
\includegraphics[width=0.48\textwidth,height=0.48\textwidth]{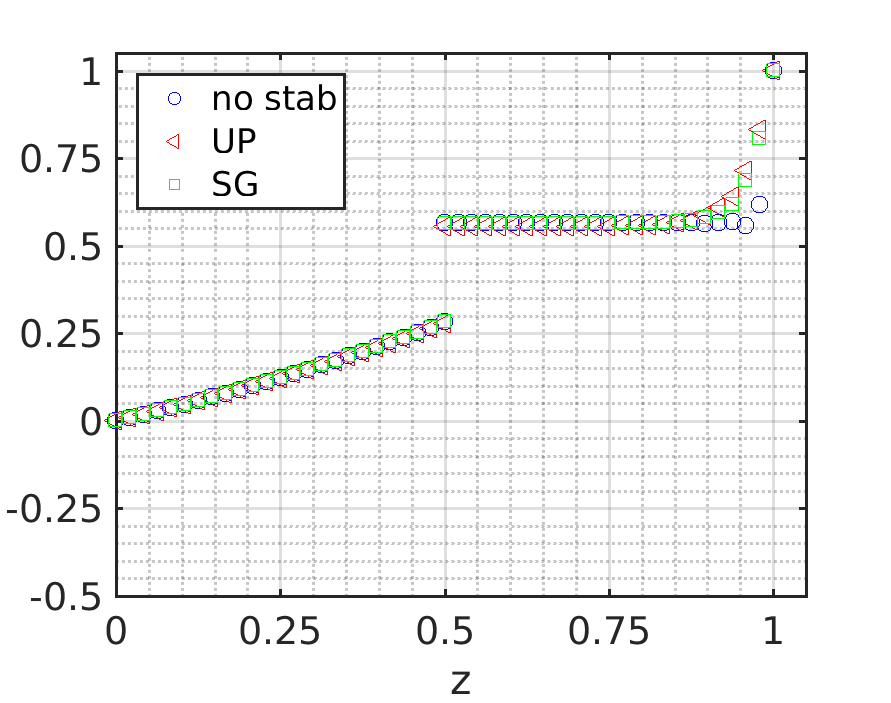}
\caption{1D plot of the computed solution $u^\ast_h$ along the $z$-axis. 
Blue circles: unstabilized solution. Green squares: SG stabilization. Red triangles: Upwind stabilization. 
$u^\ast_h$ along the $z$-axis. Left panel: $\mu=1$, $v_z = 1$ in $\Omega_1$, ${\rm Pe}_K|_{\Omega_1} = 0.0541$; 
$\mu=0.0325$, $v_z = 1$ in $\Omega_2$, ${\rm Pe}_K|_{\Omega_2} = 1.6654$. 
Right panel: $\mu=1$, $v_z = 1$ in $\Omega_1$, ${\rm Pe}_K|_{\Omega_1} = 0.0541$; 
$\mu=0.008125$, $v_z = 1$ in $\Omega_2$, ${\rm Pe}_K|_{\Omega_2} = 6.6618$.}
\label{fig:inst_disc}
\end{figure}

\section{Conclusions and perspectives}\label{sec:conclusions}

In this work we have proposed, analyzed and numerically validated 
a novel dual mixed hybrid (DMH) finite element method (FEM),
based on the Raviart-Thomas finite element space of lowest order (RT0), 
for the numerical approximation of a boundary value problem with diffusive, 
advective and reactive terms to be solved
in a three-dimensional domain with transmission conditions across a selective interface.

The new formulation combines in a unified framework 
a pair of Lagrange multipliers
to account for the interface conditions, with the 
dual mixed hybrid method for the weak formulation and discretization of the 
problem. To stabilize the computation against advection dominance, 
an artificial diffusion is introduced along the streamline direction, as in the SUPG method.

The resulting scheme is a flexible and robust numerical approach for the treatment of
heterogeneous problems where model coefficients may be subject to wide variations over the
partitioned computational domain and sharp discontinuities of the primal variable and of the
associated flux density may occur at the interface.

The well-posedness of the scheme is analyzed using the abstract theory of saddle-point
problems and optimal error estimates are proved with respect to the finite element discretization parameter. 
An efficient implementation of the method within the computational platform MP-FEMOS
is made possible by the use of static condensation 
to eliminate the internal variables and the Lagrange multipliers for the dual variable at the 
interface in favor of the hybrid variable and of the Lagrange multiplier for the primal variable at the 
interface. 

Extensive numerical tests demonstrate the theoretical conclusions and indicate 
that the proposed DMH-RT0 FEM scheme is accurate and stable in the presence of marked 
interface jump discontinuities of both solution and associated normal flux. In the case 
of strongly dominating advective terms, 
the proposed method is capable to accurately resolve steep boundary and/or interior layers
without introducing spurious unphysical oscillations or excessive smearing of the solution front.

Next objectives of the ongoing research activity on the proposed method include:
\begin{itemize}
\item extending the implementation to domains with multiple interfaces.
This will allow us to study more realistic physical situations such as the case of the interaction between
two cellular compartments separated by an extracellular fluid~\cite{Graham_2006} 
or the case of the design of advanced
memories in nanoelectronics in which materials are characterized by the presence of localized 
defects~\cite{weckx2013defect};
\item extending the mathematical model to nonlinear transmission conditions at the interface.
This will allow us to study more realistic physical problems such as the case of semipermeable 
membranes~\cite{Hron2011} or biochemical reactions in cellular biology~\cite{wood_1999};
\item extending the numerical approach and its application to transmission models 
to Hybridizable Discontinuous Galerkin (HDG) methods.
This will allow us to benefit from the high flexibility of the HDG computational framework, in particular
the possibility of adopting standard polynomial basis functions for \emph{both} primal and dual variables,
including the case of equal-order interpolation~\cite{Cockburn_proceedings};
\item extending the DMH numerical scheme to the case where the geometrical discretization of
the domain $\Omega$ is \emph{not} fitted with the interface $\Gamma$.
This will allow us to combine the discontinuous features of the DMH method with 
the flexible and efficient computational framework of Extended Finite Elements (XFEM), as 
recently analyzed in~\cite{Flemisch2016} in the study of XFEM-based approximation of 
flow in fractured porous media;
\item extending the theoretical analysis of the convergence of the scheme to include the 
artificial diffusion stabilization of Section~\ref{sec:stabilization}. 
This will allow us to characterize the effect of the 
perturbation term $\boldsymbol{\mu}^{\ast}_K$ in~\eqref{eq:modified_diffusion_tensor} 
on the accuracy of the method as $h$ is sufficiently small by estimating the introduced consistency error 
with the Strang Lemma~\cite[Chapter~5]{quarteroni1994numerical}. 
\end{itemize}

\section*{Acknowledgements}
The authors gratefully acknowledge Prof. Bernardo Cockburn for fruitful discussions on the subject 
of the article. Giovanna Guidoboni has been partially supported by the 
Chair Gutenberg funds of the Cercle Gutenberg (France) and by 
the Labex IRMIA (University of Strasbourg, France). Riccardo Sacco is a Member of the INdAM Research group GNCS
and has been partially supported by Micron Semiconductor Italia S.r.l.,
statement of work \#4505462139: "Modeling of tunneling and charging dynamics", contractors: 
Micron Semiconductor Italia S.r.l.; Dipartimento di Matematica Politecnico di Milano, Italy.

\appendix
\section{Generalized saddle-point problems}\label{sec:generalized_saddle_point_pbs}
In this appendix we consider the following generalized saddle-point problem: 
Find ${\tt u} \in \mathcal{V}$ and ${\tt p} \in \mathcal{Q}$ such that:
\begin{subequations}\label{eq:abstract_SP_problem}
\begin{align}
& a({\tt u}, {\tt v})  + b({\tt v}, {\tt p}) = F({\tt v}) \qquad \forall {\tt v} \in \mathcal{V}, \\
& b({\tt u}, {\tt q})  - c({\tt p}, {\tt q}) = G({\tt q}) \qquad \forall {\tt q}\in \mathcal{Q},
\end{align}
\end{subequations}
where $\mathcal{V}$ and $\mathcal{Q}$ are Hilbert spaces with norms $\| \cdot \|_{\mathcal{V}}$
and $\| \cdot \|_{\mathcal{Q}}$ whereas
$a$, $b$, $c$, $F$ and $G$ are bilinear forms and linear functionals such that:
\begin{subequations}\label{eq:bilinear_and_linear_forms}
\begin{align}
& a({\tt u}, {\tt v}): \mathcal{V} \times \mathcal{V} \rightarrow \mathbb{R}, \label{eq:form_a} \\
& b({\tt u}, {\tt q}): \mathcal{V} \times \mathcal{Q} \rightarrow \mathbb{R}, \label{eq:form_b} \\
& c({\tt p}, {\tt q}): \mathcal{Q} \times \mathcal{Q} \rightarrow \mathbb{R}, \label{eq:form_c} \\
& F({\tt v}): \mathcal{V} \rightarrow \mathbb{R}, \label{eq:form_F} \\
& G({\tt q}): \mathcal{Q} \rightarrow \mathbb{R}, \label{eq:form_G}
\end{align}
In the remainder of the section, denote by 
\begin{align}
& \mathcal{V}^{0} = \left\{ {\tt v} \in \mathcal{V}, \, \, 
b({\tt v}, {\tt q})= 0 \, \, \forall {\tt q} \in \mathcal{Q} \right\} & \label{eq:kernel_space}
\end{align}
the kernel of the bilinear form $b$.
\end{subequations}
In the following we address the well-posedness analysis of~\eqref{eq:abstract_SP_problem} and of its 
Galerkin finite element approximation. 
We refer to~\cite{brezzifortin1991},~\cite{RobertsThomas1991},~\cite{quarteroni1994numerical} 
and~\cite{Boffi2013} for further details and examples.

\subsection{The continuous case}\label{sec:saddle_point_continuous}
Several theoretical results for establishing the well-posedness of~\eqref{eq:abstract_SP_problem}
have been obtained in the case $c=0$. 
The case $c\neq 0$  is treated in~\cite{brezzifortin1991}
and, in more detail, in~\cite{Boffi2013}. 
A more general setting including two bilinear forms $b_1$, $b_2$, with $b_1 \neq b_2$,
is studied in~\cite{ciarlet_Jr_2003}.
A further extension is developed in~\cite{Lamichhane2008}, where the authors analyze 
a Petrov-Galerkin formulation. 
The following result can be derived from Theorem 2 of~\cite{Lamichhane2008}.

\vspace{.3cm}

\noindent\textbf{Theorem A.1. (Existence and uniqueness of solutions to generalized saddle-point problems)}
Let us consider the generalized saddle-point problem~\eqref{eq:abstract_SP_problem}.
Assume that there exist positive constants $M_a$, $M_b$ and $M_c$ such that:
\begin{subequations}\label{eq:assumptions_continuity_forms}
\begin{align}
& |a({\tt u}, {\tt v})| \leq M_a \| {\tt u} \|_{\mathcal{V}} \| {\tt v} \|_{\mathcal{V}}
\qquad {\tt u}, {\tt v} \in \mathcal{V}, \label{eq:continuity_a} \\
& |b({\tt v}, {\tt q})| \leq M_b \| {\tt v} \|_{\mathcal{V}} \| {\tt q} \|_{\mathcal{Q}}
\qquad {\tt v} \in \mathcal{V}, \, {\tt q} \in \mathcal{Q}, \label{eq:continuity_b} \\
& |c({\tt p}, {\tt q})| \leq M_c \| {\tt p} \|_{\mathcal{Q}} \| {\tt q} \|_{\mathcal{Q}}
\qquad {\tt p}, {\tt q} \in \mathcal{Q}. \label{eq:continuity_c}
\end{align}
\end{subequations}
Assume also that: 
\begin{subequations}\label{eq:assumptions_coercivity_forms}
\begin{enumerate}
\item there exists a positive constant $k_a$ such that
\begin{align}
& a({\tt u}, {\tt u}) \geq k_a \| {\tt u} \|^2_{\mathcal{V}} \qquad \forall 
{\tt u} \in \mathcal{V}^0; \label{eq:weak_coercivity_a}
\end{align}
\item there exists a positive constant $k_b$ such that
\begin{align}
& \forall {\tt q} \in \mathcal{Q}, \, \exists {\tt v} \in \mathcal{V}, \, 
{\tt v} \neq 0, \, {\rm such \, \, that} \nonumber \\
& b({\tt v}, {\tt q}) \geq k_b \| {\tt v} \|_{\mathcal{V}} \| {\tt q} \|_{\mathcal{Q}};
\label{eq:weak_coercivity_b}
\end{align}

\item the following "smallness relation" holds
\begin{align}
& \delta:= M_a \left( 1 + \frac{M_a}{k_a} \right) \frac{M_c}{k_b^2} < 1.
\label{eq:smallness_constants}
\end{align}
\end{enumerate}
\end{subequations}
Then, problem~\eqref{eq:abstract_SP_problem} has a unique solution that satisfies the stability estimates:
\begin{subequations}\label{eq:estimates_solution}
\begin{align}
& \| {\tt u} \|_{\mathcal{V}} \leq c_{11} \| F \|_{\mathcal{V}^{\prime}} + c_{12} \| G \|_{\mathcal{Q}^{\prime}}, 
 \label{eq:estimate_for_u} \\
& \| {\tt p} \|_{\mathcal{Q}} \leq c_{21} \| F \|_{\mathcal{V}^{\prime}} + c_{22} \| G \|_{\mathcal{Q}^{\prime}},
\label{eq:estimate_for_p}
\end{align}
where 
\begin{equation}
\|F \|_{\mathcal{V}^{\prime}} = \sup_{\substack{\tt v \in \mathcal V\\\tt v\neq 0}} \frac{| F(\tt v)|}{\|\tt v \|_{\mathcal V}},
\qquad 
 \|G \|_{\mathcal{Q}^{\prime}} = \sup_{\substack{\tt q \in \mathcal Q\\\tt q\neq 0}} \frac{| G(\tt q)|}{\|\tt q \|_{\mathcal Q}}
\end{equation}
and
\begin{align*}
& c_{11} = \Frac{1}{k_a} + \left( 1 + \Frac{M_a}{k_a} \right)^2 \Frac{M_c}{1-\delta}, & \\
& c_{12} = \Frac{1}{k_b} \left( 1 + \Frac{M_a}{k_a} \right) 
\left( 1 + \Frac{M_a M_c}{1-\delta} \left( 1 + \Frac{M_a}{k_a} \right) \right), & \\
& c_{21} = \frac{1}{k_b} \frac{1}{1-\delta} \left( 1 + \frac{M_a}{k_a} \right), & \\
& c_{22} = c_{21} \frac{M_a}{k_b}.
\end{align*}
\end{subequations}

\subsection{The approximate case}\label{sec:saddle_point_approx}
Let $\mathcal{T}_h$, $h>0$, be a family of unstructured partitions of the domain $\Omega$ into tetrahedral
elements as described in Section~\ref{sec:geometrical_discretization}. 
We denote by $\mathcal{V}_h$ and $\mathcal{Q}_h$ 
two finite-dimensional subspaces of $\mathcal{V}$ and
$\mathcal{Q}$, respectively. Both $\mathcal{V}_h$ and $\mathcal{Q}_h$ consist of piecewise polynomials
defined over the triangulation $\mathcal{T}_h$. Then, the Galerkin Finite Element (GFE) 
approximation of the abstract generalized saddle-point problem~\eqref{eq:abstract_SP_problem} reads: 
Find ${\tt u}_h \in \mathcal{V}_h$ and ${\tt p}_h \in \mathcal{Q}_h$ such that:
\begin{subequations}\label{eq:abstract_SP_problem_h}
\begin{align}
& a({\tt u}_h, {\tt v}_h)  + b({\tt v}_h, {\tt p}_h) = F({\tt v}_h) \qquad \forall {\tt v}_h \in \mathcal{V}_h, \\
& b({\tt u}_h, {\tt q}_h)  - c({\tt p}_h, {\tt q}_h) = G({\tt q}_h) \qquad \forall {\tt q}_h \in \mathcal{Q}_h.
\end{align}
In the remainder of the section, denote by 
\begin{align}
& \mathcal{V}^{0}_h = \left\{ {\tt v}_h \in \mathcal{V}_h, \, \, 
b({\tt v}_h, {\tt q}_h)= 0 \, \, \forall {\tt q}_h \in \mathcal{Q}_h \right\} 
& \label{eq:kernel_space_h}
\end{align}
the discrete kernel of the bilinear form $b$.
\end{subequations}

The following result is the discrete counterpart of Theorem A.1.

\vspace{.3cm}

\noindent\textbf{Theorem A.2. (Existence and uniqueness of the approximate solution to generalized saddle-point problems)}
Let us consider the generalized saddle-point problem~\eqref{eq:abstract_SP_problem}
and its GFE approximation~\eqref{eq:abstract_SP_problem_h}.
Assume that there exist positive constants $M_{a,h}$, $M_{b,h}$ and $M_{c,h}$ such that:
\begin{subequations}\label{eq:assumptions_continuity_forms_h}
\begin{align}
& |a({\tt u}_h, {\tt v}_h)| \leq M_{a,h} \| {\tt u}_h \|_{\mathcal{V}} \| {\tt v}_h \|_{\mathcal{V}}
\qquad {\tt u}_h, {\tt v}_h \in \mathcal{V}_h, \label{eq:continuity_a_h} \\
& |b({\tt v}_h, {\tt q}_h)| \leq M_{b,h} \| {\tt v}_h \|_{\mathcal{V}} \| {\tt q}_h \|_{\mathcal{Q}}
\qquad {\tt v}_h \in \mathcal{V}_h, \, {\tt q}_h \in \mathcal{Q}_h, \label{eq:continuity_b_h} \\
& |c({\tt p}_h, {\tt q}_h)| \leq M_{c,h} \| {\tt p}_h \|_{\mathcal{Q}} \| {\tt q}_h \|_{\mathcal{Q}}
\qquad {\tt p}_h, {\tt q}_h \in \mathcal{Q}_h. \label{eq:continuity_c_h}
\end{align}
\end{subequations}
Assume also that: 
\begin{subequations}\label{eq:assumptions_coercivity_forms_h}
\begin{enumerate}
\item there exists a positive constant $k_{a,h}$ such that
\begin{align}
& a({\tt u}_h, {\tt u}_h) \geq k_{a,h} \| {\tt u}_h \|^2_{\mathcal{V}} \qquad \forall 
{\tt u}_h \in \mathcal{V}^0_h; \label{eq:weak_coercivity_a_h}
\end{align}
\item there exists a positive constant $k_{b,h}$ such that
\begin{align}
& \forall {\tt q}_h \in \mathcal{Q}_h, \, \exists {\tt v}_h \in \mathcal{V}_h, \, 
{\tt v}_h \neq 0, \, {\rm such \, \, that} \nonumber \\
& b({\tt v}_h, {\tt q}_h) \geq k_{b,h} \| {\tt v}_h \|_{\mathcal{V}} \| {\tt q}_h \|_{\mathcal{Q}};
\label{eq:weak_coercivity_b_h}
\end{align}

\item the following "smallness relation" holds
\begin{align}
& \delta_h:= M_{a,h} \left( 1 + \frac{M_{a,h}}{k_{a,h}} \right) \frac{M_{c,h}}{k_{b,h}^2} < 1.
\label{eq:smallness_constants_h}
\end{align}
\end{enumerate}
\end{subequations}
Then, problem~\eqref{eq:abstract_SP_problem_h} has a unique solution that satisfies the stability estimates:
\begin{subequations}\label{eq:estimates_solution_h}
\begin{align}
& \| {\tt u}_h \|_{\mathcal{V}} \leq c_{11,h} \| F_h \|_{\mathcal{V}^{\prime}} + c_{12,h} 
\| G_h \|_{\mathcal{Q}^{\prime}}, 
 \label{eq:estimate_for_u_h} \\
& \| {\tt p}_h \|_{\mathcal{Q}} \leq c_{21,h} \| F_h \|_{\mathcal{V}^{\prime}} + c_{22,h} 
\| G_h \|_{\mathcal{Q}^{\prime}},
\label{eq:estimate_for_p_h}
\end{align}
where $F_h:= F({\tt v}_h)$, ${\tt v}_h \in \mathcal{V}_h$, and $G_h:= G({\tt q}_h)$,
${\tt q}_h \in \mathcal{Q}_h$, and
\begin{align*}
& c_{11,h} = \Frac{1}{k_{a,h}} + \left( 1 + \Frac{M_{a,h}}{k_{a,h}} \right)^2 
\Frac{M_{c,h}}{1-\delta_h}, & \\
& c_{12,h} = \Frac{1}{k_{b,h}} \left( 1 + \Frac{M_{a,h}}{k_{a,h}} \right) 
\left( 1 + \Frac{M_{a,h} M_{c,h}}{1-\delta_h} \left( 1 + \Frac{M_{a,h}}{k_{a,h}} \right) \right), & \\
& c_{21,h} = \frac{1}{k_{b,h}} \frac{1}{1-\delta_h} \left( 1 + \frac{M_{a,h}}{k_{a,h}} \right), & \\
& c_{22,h} = c_{21_h} \frac{M_{a,h}}{k_{b,h}}.
\end{align*}
\end{subequations}

\bibliographystyle{plain}
\bibliography{biblio}

\end{document}